\documentclass[utf8,a4paper]{article}
\usepackage[top=2.5cm,bottom=2.5cm,left=2.5cm,right=2.5cm]{geometry}
\usepackage{authblk}
\usepackage{amsmath}
\usepackage{amssymb}
\usepackage{amsthm}
\usepackage{color}
\usepackage{bm}
\usepackage[colorlinks,linkcolor=blue,anchorcolor=green,citecolor=red]{hyperref}
\usepackage{appendix}
\usepackage{graphicx}
\usepackage{subfigure}
\usepackage{booktabs}
\usepackage{epstopdf}

\allowdisplaybreaks[4]

\numberwithin{equation}{section}
\newtheorem{theorem}{Theorem}[section]
\newtheorem{remark}{Remark}[section]

\title{A phase field model for mass transport with semi-permeable interfaces}
\author[a,b,c]{Yuzhe Qin}
\author[e,f,*]{Huaxiong Huang} 
\author[f]{Yi Zhu}
\author[g]{Chun Liu}
\author[d,*]{Shixin Xu}
\affil[a]{Research Center for Mathematics, 
Beijing Normal University at Zhuhai, 519087, China}
\affil[b]{Research Center for Mathematics and Mathematics Education, 
Beijing Normal University at Zhuhai, 519087, China}
\affil[c]{Laboratory of Mathematics and Complex Systems (Ministry of Education), 
School of Mathematical Sciences, 
Beijing Normal University, Beijing 100875, China}
\affil[d]{Duke Kunshan University, 8 Duke Ave, Kunshan, Jiangsu, China}
\affil[e]{BNU-UIC Joint Mathematical Research center,
Beijing Normal University, Zhuhai 519087, China}
\affil[f]{Department of Mathematics and Statistics, York University,
Toronto, Ontario, Canada} 
\affil[g]{Department of Applied Mathematics, Illinois Institute of Technology, IL 60616, USA}
\affil[*]{Corresponding authors, Shixin.xu@dukekunshan.edu.cn; hhuang@uic.edu.cn}
\date{}

\begin{document} 
\maketitle  
\begin{abstract}
      In this paper,  a thermal-dynamical consistent  model for mass transfer across permeable  moving interfaces is proposed by using energy variation method.  We consider a restricted diffusion problem where the flux across the interface depends on its conductance and the difference of the concentration on each side. The diffusive interface phase-field framework used in here has several advantages over the sharp interface method. First of all, explicit tracking of the interface is no longer necessary. Secondly, the interfacial condition can be incorporated with a variable diffusion coefficient. A detailed asymptotic analysis  confirms the diffusive interface  model converges to the existing sharp interface model as the interface thickness goes to zero. A
    decoupled energy stable numerical scheme is developed to solve this system   efficiently.   Numerical simulations first illustrate the consistency of theoretical results on the sharp interface limit.  Then a convergence study and energy decay test are conducted to ensure the efficiency and stability of the numerical scheme. To illustrate the effectiveness of our phase-field approach, several examples are provided, including a study of a two-phase mass transfer problem where drops with deformable interfaces are suspended in a moving fluid.
\end{abstract}
\section{Introduction}

Mass transfer through a semi-permeable or conducting  interface 
is a common phenomenon in biology \cite{Jiang2013Cellular,gong2014immersed} and material science \cite{flynn1974mass,johansson2005mass}. 
  A representative example is that cell membranes are permeable to  oxygen \cite{wang2020immersed}, ATP \cite{Zhang2018ATP} and ions \cite{pakhomov2009lipid}. 
In this case, the domain  consists of intracellular space denoted by $\Omega^+$ and extracellular space denoted by $\Omega^-$ with a cell membrane $\Gamma$ in between (see Figure \ref{closemem} ).
\begin{figure} \label{closemem}
    \centering
    \includegraphics[width=0.5\textwidth]{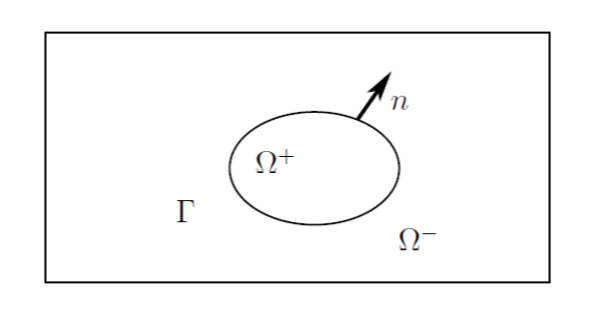}
    \caption{Schematic of mass transport  across the membrane. }
    \label{fig:my_label}
\end{figure}
The diffusion of concentration   can be described by the following diffusion equation in the bulk \cite{xu2018osmosis}, 
\begin{equation}\label{eqn: diffusion}
    \frac{\partial c^\pm }{\partial t} +\nabla\cdot(\mathbf{u}^\pm c^\pm )=\nabla\cdot(D^\pm \nabla c^\pm ), 
\end{equation}
where $c^\pm, \mathbf{u}^\pm $ represent the concentration  of the molecule and  velocity of fluid  in  $\Omega^\pm $ and $D^\pm $ is the diffusion coefficient in $\Omega^\pm $.  Here $c^\pm$ may be discontinuous cross the boundary, like the ion concentrations in  \cite{song2018electroneutral}. 
At the interface $\Gamma$, the trans-membrane flux is set to be continuous   
\begin{equation}\label{eqn: flux law}
    -D^+ \nabla c^+ \cdot \bm{n}=  -D^- \nabla c^- \cdot \bm{n}=K[\![Q(c)]\!], 
\end{equation}
where $K$ is the permeability, $[\![f]\!] = f^+-f^-$ denotes the jump of concentration across the interface and 
$\bm{n}$ is the unit normal vector of the interface. Here $Q(c)$ could have different formats, like $Q(c) = c$ in \cite{gong2014immersed} or $Q(c) = \ln{c}$ in \cite{xu2018osmosis}.



There are numerous papers investigating for mass transfer. 
As a matter of fact, the mass transfer phenomenon can be attributed to diffusive process in binary flows involving free boundaries. 
 For the standard fluid-structure interaction problems,  a variety of methods are developed over the past decades. There are two most popular classes of technique, the sharp interface methods,  like level set methods \cite{adalsteinsson1995fast}, immerse boundary methods \cite{peskin2002immersed}, immerse interface methods \cite{leveque1997immersed}, front tracking methods \cite{unverdi1992front}, and diffuse interface methods \cite{Yue2004Diffuse}.  For mass-transfer across liquid–gas interface, volume of fluid methods are proposed to solve  the phase change problem \cite{welch2000volume,davidson2002volume} and bubble
behaviours with $\mathrm{CO_2/ N_2}-$water system \cite{li2019interactions}.  The immersed interface method is extended to study water transport accross a deformable membrane by Layton \cite{Layton2006Modeling}.  

The immerse boundary method, due to its simplicity, has been applied to many fluid flow problems and has become one of the main numerical techniques for scientific computation.
Gong and Huang et al. \cite{gong2014immersed,wang2020immersed} developed a series of work to understand oxygen transport across permeable membrane by using immerse interface method. In order to ensure the restrict diffusion on the membrane \eqref{eqn: flux law},  an additional equation is introduced to describe the temporal evolution of diffusive flux. 

Unlike sharp interface methods, where the interface are  handled separately by using $\delta$ function or local reconstruction,  the diffusive interface method models the two phase flow and the interface in a  uniform way by a label function $\phi$
\cite{CH1,CH2,CH3}. 
The basic idea can be dated back to van der Waals in the late 19th century\cite{Waals1893Diffuse}. The main advantage of diffusive interface method is that follows the energy dissipation law such that the obtained model is thermal-dynamical consistent. It makes it possible to design efficient energy stable schemes for long time simulation.  However, to our best knowledge,   there is not a diffusive interface model for mass transport through a semi-permeable membrane.  The main challenge is how to impose the restricted diffusion near interface such that as interface thickness goes to zero, the sharp interface limit is consistent with boundary condition \eqref{eqn: flux law}.

In this paper,  a thermal-dynamical consistent diffusive model is first proposed by using energy variational  method \cite{shen2020energy}, which starts from two functionals for the total energy and dissipation, together with the kinematic equations based on physical laws of conservation.
The key is to modify the diffusion coefficient as a function of $\phi$  and interface permeability $K$. The restricted diffusion  only means that the changing rate of energy near the interface follows a specific dissipation rate functional.    Then following the results of Xu et al. \cite{xu2017sharp},  a detailed asymptotic analysis is presented to confirm the proposed diffusive interface model  converges to the existing sharp interface model \eqref{eqn: diffusion} and \eqref{eqn: flux law} as the interface thickness goes to zero. In the next, based on the energy dissipation law, an efficient energy stable decoupled scheme is proposed to solve the obtained system.

The structure of the paper is as follows. 
In section $\ref{sec: model}$, the phase field model for mass transport through a semi-permable interface is proposed by using energy variational method. 
In section $\ref{sec: sharp}$, the sharp interface limits of the phase field model are presented
by asymptotic analysis. 
A decoupled, linear and unconditional energy stable numerical 
scheme is developed  in  section $\ref{sec: scheme}$ by means of the stabilization method and pressure correction method. 
In section $\ref{sec: results}$,   numerical experiments are carried out to verify our theoretically 
results and study the membrane permeability effect. 
Finally, conclusions are drawn  in section $\ref{sec: conclusions}$. 
\section{Phase field model for mass transfer with hydrodynamics}\label{sec: model}
In this section, energy variation method is used to derive a diffusive interface model for mass transport through a semi-permeable membrane with restrict diffusion.
\subsection{Model derivation}
First, phase field variable $\phi$ is introduced to label the different domain, 
\begin{align}\label{def: phi}
    \begin{split}
        \phi({\bm{x},t})=\left\{
            \begin{array}{ll}
                1, &\mbox{in~}\Omega^+,\\
                -1, &\mbox{in~}\Omega^-.
            \end{array}
            \right.
    \end{split}
\end{align}
The interface between the two domains is described by the zero level set
$\Gamma_{t}=\{\bm{x}:\phi(\bm{x},t)=0\}$. 
We will omit the subscript $t$ in the following, i.e. 
we use $\Gamma$ to instead $\Gamma_{t}$.

We start from the following kinematic assumptions on the laws of conservation: in domain $\Omega$
\begin{subequations}\label{eqn: main eqns}
    \begin{align}
        &\frac{\textbf{D} \phi}{\textbf{D} t}= -\nabla\cdot \bm{j}_{\phi}, 
        \label{eqn: phase field}\\
        &\frac{\textbf{D} c}{\textbf{D} t}= -\nabla\cdot \bm{j}_{c}, 
        \label{eqn: concertraction}\\
        &\rho\frac{\textbf{D} \bm{u}}{\textbf{D} t}=
        \nabla\cdot \sigma_{\nu}+\nabla\cdot\sigma_{\phi}, 
        \label{eqn: velocity}\\
        &\nabla\cdot\bm{u}=0, 
        \label{eqn: incompressible condition}
    \end{align}
\end{subequations}
where $\rho$ is the density, $\phi$ is the phase label function, $c$ is the concentration of a substance, $\bm{u}$ is the fluid velocity, and $ \frac{\textbf{D}f}{\textbf{D}t}=\frac{\partial f}{\partial t}+(\bm{u}\cdot\nabla)f$ is the material derivative.  

Here Eq. \eqref{eqn: phase field} is the conservation of label function with unknown flux $\bm{j}_{\phi}$; Eq. \eqref{eqn: concertraction} is the conservation of mass with unknown mass flux $\bm{j}_{c}$; Eq. \eqref{eqn: velocity} is the conservation of momentum with  unknown stress induced by viscosity of fluid
$\sigma_{\nu}$ and  stress induced by two-phase flow interface  $ \sigma_{\phi} $.     

On the boundary of  the domain, the boundary conditions can be given as 
\begin{equation}\label{eqn: boundary conditions}
    \bm{j}_{\phi}\cdot\bm{n}|_{\partial\Omega}=0, \quad
    \bm{j}_{c}\cdot\bm{n}|_{\partial\Omega}=0, \quad
    \nabla\phi\cdot\bm{n}|_{\partial\Omega}=0, \quad
    \bm{u}|_{\partial\Omega}=0, 
\end{equation}
where $\bm{n}$ is the outward normal on the domain boundary $\partial\Omega$. 
Then the system $\eqref{eqn: main eqns}$ conserves the local mass density, 
i.e. $\frac{\mathrm{d}}{\mathrm{d}t}\int_{\Omega}c\mathrm{d}\bm{\bm{x}}=0$ 
and $\frac{\mathrm{d}}{\mathrm{d}t}\int_{\Omega}\phi\mathrm{d}\bm{\bm{x}}=0$. 

The total energy consists of kinetic energy, mix entropy and the phase mixing energy
\begin{equation}\label{egy: total}
    E_{total}=E_{kin}+E_{ent}+E_{mix}
    =\frac{1}{2}\int_{\Omega}\rho|\bm{u}|^{2}\mathrm{d}\bm{x}+\int_{\Omega}RTc(\ln \frac{c}{c_{0}}-1)\mathrm{d}\bm{x}
    +\int_{\Omega}\lambda(G(\phi)+\frac{\gamma^{2}}{2}|\nabla\phi|^{2})\mathrm{d}\bm{x},
\end{equation}
where $G(\phi)=\frac{1}{4}(1-\phi^{2})^{2}$ is the double well potential. 
$R$ is the gas constant number, $T$ is the thermodynamic temperature, 
$c_{0}$ is the reference concentration,
$\lambda$ is the energy density, $\gamma$ is the thickness of the interface.  

Then according to the total energy, we could define the chemical potentials 
\begin{eqnarray}
\mu_c = \frac{\delta E}{\delta c} = RT \ln\frac{c}{c_0},\label{eqn: potential_c}\\
\mu_\phi=\frac{\delta E}{\delta \phi} = -\lambda \gamma^2\Delta\phi+\lambda G'(\phi).\label{eqn: potential_phi} 
\end{eqnarray}

  The dissipative functional is  composed of the dissipation due to fluid friction, irreversible mixing of the substance and irreversible mixing of two phases in bulk
\begin{equation}\label{def: dissipation}
    \Delta=\underbrace{\int_{\Omega}2\nu(\phi)|D_{\nu}|^{2}\mathrm{d}\bm{x}}_\text{friction} +
    \underbrace{\int_{\Omega}\frac{D_{\it{eff}}c}{RT}|\nabla\mu_{c}|^{2}\mathrm{d}\bm{x}}_\text{diffusion}
    +\underbrace{\int_{\Omega}\mathcal{M}|\nabla\mu_{\phi}|^{2}\mathrm{d}\bm{x}}_\text{mixing}. 
\end{equation}
where $\nu(\phi)$ is the viscosity for the fluid,  
$D_{\nu}=\frac{\nabla\bm{u}+(\nabla\bm{u})^{T}}{2}$ is the strain rate, $\mathcal{M}$ is the phenomenological mobility.

In order to model the restricted diffusion due to the permeability  of the interface,  the diffusion  coefficient $D_{\emph{eff}}$ is modelled in the following way
\begin{equation}\label{def: Deff}
    \frac 1 {D_{\it{eff}}} =
    \frac{(1-\phi^{2})^{2}}{A Kq\gamma}
    +\frac{1-\phi}{2D_{-}}
    +\frac{1+\phi}{2D_{+}}, 
\end{equation}
where $K$ is the permeability for the membrane, $A$ is a constant to be determined, $q=\frac{dQ}{dc}$ and 
$D_{+}$ and $D_{-}$ are the diffusion coefficients in domain 
$\Omega^+$ and 
$\Omega^-$ respectively.


The energy dissipative law \cite{eisenberg2010energy,xu2018osmosis,shen2020energy} states that without external force acting on the system, the changing rate of total energy equals the dissipation
\begin{equation}\label{eqn: dissipation}
    \frac{\mathrm{d}}{\mathrm{d}t}E_{total}=-\Delta. 
\end{equation}

Then the definition of total energy functional \eqref{egy: total} yields
\begin{eqnarray}\label{eqn: dedt}
\frac{dE_{tot}}{dt} &= &\frac{dE_{ent}}{dt} +\frac{dE_{kin}}{dt} +\frac{dE_{mix}}{dt} \nonumber\\
&=& I_1 +I_2 +I_3.
\end{eqnarray}
By using the definitions of chemical potential \eqref{eqn: potential_c} and conservation law \eqref{eqn: concertraction}, the first part yields 
\begin{align}\label{eqn: entropy dissipation}
I_1
&=\int_{\Omega}\mu_{c}(\frac{\partial c}{\partial t}+\bm{u}\cdot\nabla c)\mathrm{d}\bm{x}
\nonumber\\
&=-\int_{\Omega}\mu_{c}\nabla\cdot\bm{j}_{c}\mathrm{d}\bm{x}
\nonumber\\
&=-\int_{\partial\Omega}\mu_{c}\bm{j}_{c}\cdot\bm{n}\mathrm{d}\bm{x}
+\int_{\Omega}\nabla\mu_{c}\cdot\bm{j}_{c}\mathrm{d}\bm{x}
\nonumber\\
&=\int_{\Omega}\nabla\mu_{c}\cdot\bm{j}_{c}\mathrm{d}\bm{x},
\end{align}
where the nonflux boundary condition for $\bm{j}_c$ is used here. 
 
 For the second term, by using Eqs.\eqref{eqn: velocity} and \eqref{eqn: incompressible condition},  we have 
\begin{align}\label{eqn: kinetic dissipation}
\frac{\mathrm{d}}{\mathrm{d}t}E_{kin}
=&\frac{\mathrm{d}}{\mathrm{d}t}(\frac{1}{2}\int_{\Omega}\rho|\bm{u}|^{2}\mathrm{d}\bm{x})
\nonumber\\
=&\frac{1}{2}\int_{\Omega}\frac{\partial\rho}{\partial t}|\bm{u}|^{2}\mathrm{d}\bm{x}
+\int_{\Omega}\rho\bm{u}\cdot\frac{\partial \bm{u}}{\partial t}\mathrm{d}\bm{x}
\nonumber\\
=&\frac{1}{2}\int_{\Omega}\frac{\partial\rho}{\partial t}|\bm{u}|^{2}\mathrm{d}\bm{x}
+\int_{\Omega}\rho\bm{u}\cdot\frac{\textbf{D}\bm{u}}{\textbf{D}t}\mathrm{d}\bm{x}
+\int_{\Omega}\nabla\cdot(\rho\bm{u})\frac{|\bm{u}|^{2}}{2}\mathrm{d}\bm{x}
\nonumber\\
=&\int_{\Omega}\bm{u}\cdot(\nabla\cdot\sigma_{\eta}+\nabla\cdot\sigma_{\phi})\mathrm{d}\bm{x}
-\int_{\Omega}p\nabla\cdot\bm{u}\mathrm{d}\bm{x}
\nonumber\\
=&-\int_{\Omega}\nabla\bm{u}:(\sigma_{\eta}+\sigma_{\phi})\mathrm{d}\bm{x}
-\int_{\Omega}p\nabla\cdot\bm{u}\mathrm{d}\bm{x}. 
\end{align}
where pressure is induced as a Lagrange multiplier for incompressibility. 

The last term is calculated with Eqs\eqref{eqn: phase field} and \eqref{eqn: potential_phi}   
\begin{align}\label{eqn: mixing dissipation}
\frac{d}{dt}E_{mix} = &\frac{\mathrm{d}}{\mathrm{d}t}\int_{\Omega}(G(\phi)+\frac{\gamma^{2}}{2}|\nabla\phi|^{2})\mathrm{d}\bm{x}
\nonumber\\
=&\int_{\Omega}(G^{\prime}(\phi)\frac{\partial\phi}{\partial t}+\gamma^{2}\nabla\phi\cdot\nabla\frac{\partial\phi}{\partial t})\mathrm{d}\bm{x}
\nonumber\\
=&\int_{\Omega}G^{\prime}(\phi)\frac{\partial\phi}{\partial t}\mathrm{d}\bm{x}+\int_{\partial\Omega}\gamma^{2}\frac{\partial\phi}{\partial\bm{n}}\frac{\partial\phi}{\partial t}\mathrm{d}s-\int_{\Omega}\nabla\cdot(\gamma^{2}\nabla\phi)\frac{\partial\phi}{\partial t}\mathrm{d}\bm{x}
\nonumber\\
=&\int_{\Omega}(G^{\prime}(\phi)-\nabla\cdot(\gamma^{2}\nabla\phi))\frac{\partial\phi}{\partial t}\mathrm{d}\bm{x}
\nonumber\\
=&\int_{\Omega}\mu_{\phi}(\frac{\partial\phi}{\partial t}+\bm{u}\cdot\nabla\phi)\mathrm{d}\bm{x}
-\int_{\Omega}\mu_{\phi}\bm{u}\cdot\nabla\phi\mathrm{d}\bm{x}
\nonumber\\
=&-\int_{\Omega}\mu_{\phi}\nabla\cdot\bm{j}_{\phi}\mathrm{d}\bm{x}
-\int_{\Omega}\mu_{\phi}\bm{u}\cdot\nabla\phi\mathrm{d}\bm{x}
\nonumber\\
=&\int_{\Omega}\nabla\mu_{\phi}\cdot\bm{j}_{\phi}\mathrm{d}\bm{x}
-\int_{\partial\Omega}\mu_{\phi}\bm{j}_{\phi}\cdot\bm{n}\mathrm{d}s
-\int_{\Omega}(G(\phi)-\nabla\cdot(\gamma^{2}\nabla\phi))\bm{u}\cdot\nabla\phi\mathrm{d}\bm{x}
\nonumber\\
=&\int_{\Omega}\nabla\mu_{\phi}\cdot\bm{j}_{\phi}\mathrm{d}\bm{x}
-\int_{\Omega}\bm{u}\cdot \nabla G(\phi)\mathrm{d}\bm{x}
+\int_{\Omega}\gamma^{2}\Delta\phi\nabla\phi\cdot\bm{u}\mathrm{d}\bm{x}
\nonumber\\
=&\int_{\Omega}\nabla\mu_{\phi}\cdot\bm{j}_{\phi}\mathrm{d}\bm{x}
-\int_{\Omega}\bm{u}\cdot \nabla G(\phi)\mathrm{d}\bm{x}
+\int_{\Omega}\gamma^{2}(\nabla\cdot(\nabla\phi\otimes\nabla\phi)
-\frac{1}{2}\nabla|\nabla\phi|^{2})\cdot\bm{u}\mathrm{d}\bm{x}
\nonumber\\
=&\int_{\Omega}\nabla\mu_{\phi}\cdot\bm{j}_{\phi}\mathrm{d}\bm{x}
-\int_{\Omega}\bm{u}\cdot \nabla (G(\phi)+\frac{\gamma^{2}}{2}|\nabla\phi|^{2})\mathrm{d}\bm{x}
+\int_{\Omega}\gamma^{2}\nabla\cdot(\nabla\phi\otimes\nabla\phi)\cdot\bm{u}\mathrm{d}\bm{x}
\nonumber\\
=&\int_{\Omega}\nabla\mu_{\phi}\cdot\bm{j}_{\phi}\mathrm{d}\bm{x}
-\int_{\Omega}\nabla\cdot(\bm{u}(G(\phi)+\frac{\gamma^{2}}{2}|\nabla\phi|^{2}))\mathrm{d}\bm{x}
+\int_{\Omega}(G(\phi)+\frac{\gamma^{2}}{2}|\nabla\phi|^{2})\nabla\cdot\bm{u}\mathrm{d}\bm{x}
\nonumber\\
&+\int_{\Omega}\gamma^{2}\nabla\cdot(\nabla\phi\otimes\nabla\phi)\cdot\bm{u}\mathrm{d}\bm{x}
\nonumber\\
=&\int_{\Omega}\nabla\mu_{\phi}\cdot\bm{j}_{\phi}\mathrm{d}\bm{x}
-\int_{\partial\Omega}(G(\phi)+\frac{\gamma^{2}}{2}|\nabla\phi|^{2}))\bm{u}\cdot\bm{n}\mathrm{d}s
+\int_{\Omega}\gamma^{2}\nabla\cdot(\nabla\phi\otimes\nabla\phi)\cdot\bm{u}\mathrm{d}\bm{x}
\nonumber\\
=&\int_{\Omega}\nabla\mu_{\phi}\cdot\bm{j}_{\phi}\mathrm{d}\bm{x}
-\int_{\Omega}\gamma^{2}(\nabla\phi\otimes\nabla\phi):\nabla\bm{u}\mathrm{d}\bm{x},
\end{align}
where the nonflux boundary of $\bm{j}_{\phi}$ is used. 

Then  substituting above three equations to Eq. \eqref{eqn: dedt} yields 
\begin{align}\label{eqn: EVA}
    &\frac{\mathrm{d}}{\mathrm{d}t}E_{total}
    =\frac{\mathrm{d}}{\mathrm{d}t}E_{ent}
    +\frac{\mathrm{d}}{\mathrm{d}t}E_{kin}
    +\frac{\mathrm{d}}{\mathrm{d}t}E_{mix}
    \nonumber\\
    =&\int_{\Omega}\nabla\mu_{c}\cdot\bm{j}_{c}\mathrm{d}\bm{x}
    -\int_{\Omega}\nabla\bm{u}:(\sigma_{\nu}+\sigma_{\phi})\mathrm{d}\bm{x}
    -\int_{\Omega}p\nabla\cdot\bm{u}\mathrm{d}\bm{x}
    +\int_{\Omega}\nabla\mu_{\phi}\cdot\bm{j}_{\phi}\mathrm{d}\bm{x}
    -\int_{\Omega}\gamma^{2}(\nabla\phi\otimes\nabla\phi):\nabla\bm{u}\mathrm{d}\bm{x}
    \nonumber\\
    =&-\int_{\Omega}\frac{D_{\it{eff}}c}{RT}|\nabla\mu_{c}|^{2}\mathrm{d}\bm{x}
    -\int_{\Omega}\mathcal{M}|\nabla\mu_{\phi}|^{2}\mathrm{d}\bm{x}
    -\int_{\Omega}2\nu|D_{\nu}|^{2}\mathrm{d}\bm{x}. 
\end{align}
Comparing two sides of the above equation $\eqref{eqn: EVA}$ 
and calculating the functional deriative,  we obtain 
\begin{subequations}\label{eqn: flux}
    \begin{align}
        &\bm{j}_{\phi}=-\mathcal{M}\nabla\mu_{\phi}, 
        \label{eqn: flux_phi}\\
        &\bm{j}_{c}=-\frac{D_{\it{eff}}c}{RT}\nabla\mu_{c}, 
        \label{eqn: flux_c}\\
        &\sigma_{\nu}=2\nu D_{\nu}-p\textbf{I}, 
        \label{eqn: tensor_nu}\\
        &\sigma_{\phi}=-\lambda\gamma^{2}(\nabla\phi\otimes\nabla\phi).
        \label{eqn: tensor_phi}
    \end{align}
\end{subequations}

Then   the diffusive interface model for mass transport through semi-permeable membrane with restricted diffusion is summarized as follows
\begin{subequations}\label{eqn: main}
    \begin{align}
        &\frac{\textbf{D} \phi}{\textbf{D} t}= -\nabla\cdot\mathcal{M}\nabla \mu, 
        \label{eqn: phi}\\
        &\frac{\textbf{D} c}{\textbf{D} t}= -\nabla\cdot \bm{j}, 
        \label{eqn: c}\\
        &\rho\frac{\textbf{D} \bm{u}}{\textbf{D} t}= 
        \nabla\cdot(\nu\nabla\bm{u})-\nabla p
        -\lambda\gamma^{2}\nabla\cdot(\nabla\phi\otimes\nabla\phi), 
        \label{eqn: u}\\
        &\nabla\cdot\bm{u}=0, 
        \label{eqn: nabla_u}\\
        &\mu=-\lambda\gamma^{2}\nabla^{2}\phi+\lambda G^{\prime}(\phi), 
        \label{eqn: mu_phi}\\
        &\bm{j}=-D_{\it{eff}}\nabla c, 
        \label{eqn: j_c}\\
        &D_{\it{eff}}^{-1}=
        \frac{(1-\phi^{2})^{2}}{A K q\gamma}
        +\frac{1-\phi}{2D_{-}}
        +\frac{1+\phi}{2D_{+}}, 
        \label{eqn: D_eff}
    \end{align}
\end{subequations}
with boundary conditions  
\begin{equation}\label{eqn: bdc}
    \nabla\phi\cdot\bm{n}|_{\partial\Omega}=0, \quad
    \nabla c\cdot\bm{n}|_{\partial\Omega}=0, \quad
    \nabla\mu\cdot\bm{n}|_{\partial\Omega}=0, \quad
    \bm{u}|_{\partial\Omega}=0.
\end{equation} 

\subsection{Non-dimensionalization}

Now we introduce the dimensionless variables 
\begin{equation}\label{variables}
    \tilde{\bm{x}}=\frac{\bm{x}}{L}, \quad 
    \tilde{\bm{u}}=\frac{\bm{u}}{U}, \quad 
    \tilde{t}=\frac{t}{T}, \quad 
    \tilde{c}=\frac{c}{c_{0}}, \quad
    \tilde{K}=\frac{KL}{D_0},\quad
    \tilde{D} = \frac{D}{D_0},\quad 
    T=\frac{L}{U}. 
\end{equation} 
Here $L, ~T,~c_0,~U ~D_0 $ are  the characteristic length, time, concentration  velocity and diffusion constant.
For convenience, the tilde symbol will be removed in the dimensionless quantities, and the dimensionless system of  \eqref{eqn: main}  is given by

\begin{subequations}\label{eqn: dimensionaless system}
    \begin{align}
        &Re\big(\frac{\partial\bm{u}}{\partial t}
        +(\bm{u}\cdot\nabla)\bm{u}\big)+\nabla p
        =\nabla^{2}\bm{u}+\frac{1}{Ca}\mu\nabla\phi,
        \label{eqn: u}\\
        &\nabla\cdot\bm{u}=0, 
        \label{eqn: div_u}\\
        &\frac{\partial \phi}{\partial t}+(\bm{u}\cdot\nabla)\phi 
        =\nabla\cdot(\mathcal{M}\nabla\mu),
        \label{eqn: phi}\\
        &\mu=-\epsilon\nabla^{2}\phi+\frac{1}{\epsilon}\phi(\phi^{2}-1), 
        \label{eqn: mu_phi}\\
        &\frac{\partial c}{\partial t}+(\bm{u}\cdot\nabla)c
        =-\frac{1}{Pe}\nabla\cdot\bm{j},
        \label{eqn: c}\\
        &\bm{j}=-D_{\it{eff}}\nabla c, 
        \label{eqn: flux}\\
        & D_{\it{eff}}^{-1}=\frac{(1-\phi^{2})^{2}}{A K q \epsilon}
        +\frac{1-\phi}{2D^{-}}
        +\frac{1+\phi}{2D^{+}},
        \label{eqn: diffuse_coef}
    \end{align}
\end{subequations}
with the boundary condition 
\begin{equation}\label{eqn: bdc}
    \bm{u}|_{\partial\Omega}=0,\quad 
    \nabla\phi\cdot\bm{n}|_{\partial\Omega}=0,\quad 
    \nabla\mu\cdot\bm{n}|_{\partial\Omega}=0,\quad 
    \nabla c\cdot\bm{n}|_{\partial\Omega}=0. 
\end{equation}
with the dimensionless   parameters
\begin{equation}\label{num: Re_Ca}
Re=\frac{\rho UL}{\nu}, \quad 
Ca=\frac{\nu U}{\lambda\gamma}, \quad 
Pe=\frac{UL}{D_0}.
\end{equation}
\begin{remark}
During the dimensionless, we use the fact  that 
\begin{align}
    &\nabla\cdot(\nabla\phi\otimes\nabla\phi)
     \nonumber\\
    =&\nabla^{2}\phi\nabla\phi+\frac{1}{2}\nabla|\nabla\phi|^{2}
     \nonumber\\
     =&-\mu\nabla\phi+\frac{1}{2}\nabla\Big(|\nabla\phi|^{2}+\frac{1}{\epsilon}G(\phi)\Big), \nonumber
 \end{align}
 and redefine pressure function 
 \begin{equation}
     p=p+\frac{\epsilon}{2Ca}|\nabla\phi|^{2}+\frac{1}{2Ca}G(\phi). \nonumber
 \end{equation}
\end{remark}

\begin{theorem}
If $\phi$, $c$, $\bm{u}$ and $p$ are smooth solutions of the system $\eqref{eqn: dimensionaless system}$ 
with boundary condition $\eqref{eqn: bdc}$, then the following energy law is satisfied: 
\begin{align}\label{eqn: energy law}
    \frac{\mathrm{d}}{\mathrm{d}t}\mathcal{E}_{total}
    =-\int_{\Omega}\big(|\nabla\bm{u}|^{2}+\frac{\mathcal{M}}{Ca}|\nabla\mu|^{2}+D_{\it{eff}}c|\nabla\mu_{c}|^{2}\big)\mathrm{d}\bm{x}, 
\end{align}
where 
\begin{align}
    \mathcal{E}_{total}=\int_{\Omega}\Big(\frac{Re}{2}|\bm{u}|^{2}
    +\frac{1}{Ca}\big(\frac{\epsilon}{2}|\nabla\phi|^{2}+\frac{1}{4\epsilon}(1-\phi^{2})^{2}\big)
    +c\ln c\Big)\mathrm{d}\bm{x}.
\end{align}
\end{theorem}

\begin{proof}
By taking the $L^{2}$ inner product of $\eqref{eqn: u}$, $\eqref{eqn: phi}$, $\eqref{eqn: mu_phi}$, 
$\eqref{eqn: c}$ and $\eqref{eqn: flux}$ with $\bm{u}$, $\mu$, $\frac{\partial \phi}{\partial t}$, 
$\mu_{c}$ and$\frac{\partial c}{\partial t}$ respectively, one obtains immediately $\eqref{eqn: energy law}$. 
\end{proof}
\section{Sharp interface limit}\label{sec: sharp}

In this section, a detailed asymptotic analysis is presented by using the results in \cite{xu2017sharp} to show the sharp interface limits of the obtained system  \eqref{eqn: dimensionaless system}. Here we assume the viscosity is constant and   $\mathcal{M} = \alpha\epsilon $, where $\alpha$ is a constant independent of $\epsilon$.

\subsection{Outer expansions} 
Far from the two-phase interface $\Gamma$, we use the following ansatz:
\begin{subequations}\label{eqns: outer expansions}
    \begin{align}
        &\bm{u}_{\epsilon}^{\pm}=
        \bm{u}_{0}^{\pm}+\epsilon\bm{u}_{1}^{\pm}+\epsilon^{2}\bm{u}_{2}^{\pm}+o(\epsilon^{2}),
        \label{out: u}\\
        &p_{\epsilon}^{\pm}=
        p_{0}^{\pm}+\epsilon p_{1}^{\pm}+\epsilon^{2} p_{2}^{\pm}+o(\epsilon^{2}),
        \label{out: p}\\
        &\phi_{\epsilon}^{\pm}=
        \phi_{0}^{\pm}+\epsilon \phi_{1}^{\pm}+\epsilon^{2} \phi_{2}^{\pm}+o(\epsilon^{2}), 
        \label{out: phi}\\
        &\mu_{\epsilon}^{\pm}=
        \epsilon^{-1}\mu_{0}^{\pm}+\mu_{1}^{\pm}+\epsilon \mu_{2}^{\pm}+o(\epsilon). 
        \label{out: mu_phi}\\
        &c_{\epsilon}^{\pm}=
        c_{0}^{\pm}+\epsilon c_{1}^{\pm}+\epsilon^{2} c_{2}^{\pm}+o(\epsilon^{2}), 
        \label{out: c}\\
        &\bm{j}_{\epsilon}^{\pm}=
        \bm{j}_{0}^{\pm}+\epsilon \bm{j}_{1}^{\pm}+\epsilon^{2} \bm{j}_{2}^{\pm}+o(\epsilon^{2}), 
        \label{out: j}
    \end{align}
\end{subequations}
Sine the concentration $c_{\epsilon}^{\pm}$  does not effect equations of $\phi$ and $\bm{u}$ ,   we directly call  the results in \cite{xu2017sharp},  
\begin{subequations}
 \begin{align}
     \phi_{0}^{\pm}= C_0^\pm,\quad\mbox{in}~ \Omega^{\pm},\\
     \mu_0^\pm = (C_0^\pm)^3-C_0^\pm,\label{out: mu}  \\
       \nabla\cdot\bm{u}_{0}^{\pm}=0,\label{out: nabla_u}\\ 
         Re\big(\frac{\partial\bm{u}_{0}^{\pm}}{\partial t}
    +(\bm{u}_{0}^{\pm}\cdot\nabla)\bm{u}_{0}^{\pm}\big)
    +\nabla p_{0}^{\pm}
    =\nabla^{2}\bm{u}_{0}^{\pm}
    +\frac{1}{Ca}\mu_{0}^{\pm}\nabla\phi_{1}^{\pm}. \label{out: ns}
      \end{align}
\end{subequations}

Substituting $\eqref{out: c}$,$\eqref{out: j}$  
into Eq. \eqref{eqn: c} yields 
we have 
\begin{align}\label{out: j1}
    &[2D^{+}D^{-}(1-(\phi_{0}^{\pm}+\epsilon \phi_{1}^{\pm}
    +\epsilon^{2} \phi_{2}^{\pm}+o(\epsilon^{2}))^{2})^{2}
    +A K q \epsilon (D^{+}+D^{-})
    \nonumber\\
    &
    \qquad +A K q \epsilon (D^{-}-D^{+})(\phi_{0}^{\pm}+\epsilon \phi_{1}^{\pm}
    +\epsilon^{2} \phi_{2}^{\pm}+o(\epsilon^{2}))]
    (\bm{j}_{0}^{\pm}+\epsilon \bm{j}_{1}^{\pm}
    +\epsilon^{2} \bm{j}_{2}^{\pm}+o(\epsilon^{2}))
    \nonumber\\
    &
    =-2A K \epsilon D^{+}D^{-}
    [q(c_{0}^{\pm})
    +\epsilon q^{\prime}(c_{0}^{\pm})c_{1}^{\pm}
    +\epsilon^{2}(\frac{1}{2}q^{\prime\prime}(c_{0}^{\pm})(c_{1}^{\pm})^{2}
    +q^{\prime}(c_{0}^{\pm})c_{2}^{\pm})
    +o(\epsilon^{2})]
    \nabla (c_{0}^{\pm}+\epsilon c_{1}^{\pm}
    +\epsilon^{2} c_{2}^{\pm}+o(\epsilon^{2})), 
\end{align}
The leading order term of the above equation  is
\begin{equation}
    (1-(\phi_{0}^{\pm})^{2})^{2}\bm{j}_{0}^{\pm}=0, 
\end{equation}
and the next order term is 
\begin{equation}\label{eq: c0j0}
    A Kq(c_{0}^{\pm})[(D^{+}+D^{-})+(D^{-}-D^{+})\phi_{0}^{\pm}]\bm{j}_{0}^{\pm}
    +2D^{+}D^{-}(1-(\phi_{0}^{\pm})^{2})^{2}\bm{j}_{1}^{\pm}
    =-2A K D^{+}D^{-}q(c_{0}^{\pm})\nabla c_{0}^{\pm}.
\end{equation}

The leading order of \eqref{eqn: c} gives us that 
\begin{equation}
    \frac{\partial c^{\pm}_{0}}{\partial t}+\bm{u}^{\pm}_{0}\cdot\nabla c^{\pm}_{0}
    =-\frac{1}{Pe}\nabla\cdot\bm{j}^{\pm}_{0}. 
\end{equation}
 
\subsection{Inner expansions}
We first introduce the signed distance function $d\left(\bm{x}\right)$ to the interface $\Gamma$. 
Immediately, we have $\nabla d=\bm{n}$. 
After defining a new rescaled variable 
\begin{equation}
    \xi=\frac{d\left(\bm{x}\right)}{\epsilon}, 
\end{equation}
for any scalar function $f\left(\bm{x}\right)$, we can rewrite it as 
\begin{equation}
    f\left(\bm{x}\right)=\tilde{f}\left(\bm{x},\xi\right), 
\end{equation}
and the relevant operators are
\begin{subequations}
    \begin{align}
        &\nabla f\left(\bm{x}\right)=
        \nabla_{\bm{x}}\tilde{f}+\epsilon^{-1}\partial_{\xi}\tilde{f}\bm{n},
        \\
        &\nabla^{2} f\left(\bm{x}\right)
        =\nabla^{2}_{\bm{x}}\tilde{f}+\epsilon^{-1}\partial_{\xi}\tilde{f}K
        +2\epsilon^{-1}\left(\bm{n}\cdot\nabla_{\bm{x}}\right)\partial_{\xi}\tilde{f}
        +\epsilon^{-2}\partial_{\xi\xi}\tilde{f},
        \\
        &\partial_{t}f=\partial_{t}\tilde{f}+\epsilon^{-1}\partial_{\xi}\tilde{f}\partial_{t}d, 
    \end{align}
\end{subequations}
and for a vector function $\tilde{\bm{g}}\left(\bm{x}\right)$, we have 
\begin{equation}
    \nabla \cdot \tilde{\bm{g}}\left(\bm{x}\right)=
    \nabla_{\bm{x}}\cdot\tilde{\bm{g}}+\epsilon^{-1}\partial_{\xi}\tilde{\bm{g}}\cdot\bm{n},
\end{equation}

Here the $\nabla_{\bm{x}}$ and $\nabla^{2}_{\bm{x}}$ stand for the gradient and Laplace 
with respect to $\bm{x}$, respectively. 
And we use the fact that $\nabla_{\bm{x}}\cdot\bm{n}=\kappa$. 
$\kappa\left(\bm{x}\right)$ for $\bm{x}\in\Gamma\left(t\right)$ is 
the mean curvature of the interface.
In the inner region, we assume that 
\begin{subequations}\label{eqn: inner expansions}
    \begin{align}
        &\tilde{\bm{u}}_{\epsilon}
        =\tilde{\bm{u}}_{0}+\epsilon\tilde{\bm{u}}_{1}
        +\epsilon^{2}\tilde{\bm{u}}_{2}+o\left(\epsilon^{2}\right),
        \\
        &\tilde{\phi}_{\epsilon}
        =\tilde{\phi}_{0}+\epsilon\tilde{\phi}_{1}
        +\epsilon^{2}\tilde{\phi}_{2}+o\left(\epsilon^{2}\right),
        \\
        &\tilde{p}_{\epsilon}
        =\tilde{p}_{0}+\epsilon\tilde{p}_{1}
        +\epsilon^{2}\tilde{p}_{2}+o\left(\epsilon^{2}\right),
        \\
        &\tilde{c}_{\epsilon}
        =\tilde{c}_{0}+\epsilon\tilde{c}_{1}
        +\epsilon^{2}\tilde{c}_{2}+o\left(\epsilon^{2}\right),
        \\
        &\tilde{\bm{j}}_{\epsilon}
        =\tilde{\bm{j}}_{0}+\epsilon\tilde{\bm{j}}_{1}
        +\epsilon^{2}\tilde{\bm{j}}_{2}+o\left(\epsilon^{2}\right),
        \\
        &\tilde{\mu}_{\epsilon}
        =\epsilon^{-1}\tilde{\mu}_{0}+\tilde{\mu}_{1}
        +\epsilon\tilde{\mu}_{2}+o\left(\epsilon\right). 
    \end{align}
\end{subequations}
And we have 
\begin{equation}\label{in: h}
    q(\tilde{c}_{\epsilon}^{\pm})
    =q(\tilde{c}_{0}^{\pm})
    +\epsilon q^{\prime}(\tilde{c}_{0}^{\pm})\tilde{c}_{1}^{\pm}
    +\epsilon^{2}(\frac{1}{2}q^{\prime\prime}(\tilde{c}_{0}^{\pm})(\tilde{c}_{1}^{\pm})^{2}
    +q^{\prime}(\tilde{c}_{0}^{\pm})\tilde{c}_{2}^{\pm})
    +o(\epsilon^{2}).
\end{equation}

Then the results in \cite{xu2017sharp} show that

\begin{equation}
    \tilde{\phi}_{0}=\tanh\frac{\xi}{\sqrt{2}}. 
\end{equation}
By the matching condition $\lim\limits_{\xi\to\pm\infty}\tilde{\phi}_{0}=\phi_{0}^{\pm}$,
we have 
\begin{equation}\label{eq: phi0}
    \phi_{0}^{\pm}=\pm 1,~\mbox{in}~\Omega^{\pm}. 
\end{equation}
This will lead to $\mu_{0}^{\pm}=0$. 
Therefore the equation $\eqref{out: ns}$ is reduced to 
\begin{equation}
    Re\big(\frac{\partial\bm{u}_{0}^{\pm}}{\partial t}
    +(\bm{u}_{0}^{\pm}\cdot\nabla)\bm{u}_{0}^{\pm}\big)
    +\nabla p_{0}^{\pm}
    =\nabla^{2}\bm{u}_{0}^{\pm}. 
\end{equation}

Also, thanks to Eq. \eqref{eq: phi0},  Eq. \eqref{eq: c0j0} yields 

\begin{equation}
    \bm{j}_{0}^{\pm}=-D^{\pm}\nabla c_{0}^{\pm}. 
\end{equation}
with interface conditions
\begin{subequations}
    \begin{align}
     \partial_{t}d+\bm{u}_{0}\cdot\bm{n}=0\label{eq:dcons}\\ V_{n}=\bm{u}_{0}\cdot\bm{n},\\
        [\![\bm{u}_{0}]\!]=0,\\
            [\![-p_{0}\bm{n}+(\bm{n}\cdot\nabla)\bm{u}_{0}]\!]=\frac{1}{Ca}\sigma\kappa\bm{n},
    \end{align}
\end{subequations}
where $V_n$ isthe normal velocity of the interface $\Gamma$, $
    \sigma=\int_{-\infty}^{\infty}(\partial_{\xi}\tilde{\phi}_{0})^{2}\mathrm{d}\xi
    =\frac{2\sqrt{2}}{3} $ is the surface tension.

The concentration function can be transformed into the following form: 
\begin{subequations}
    \begin{align}
        &\partial_{t}\tilde{c}_{\epsilon}
        +\epsilon^{-1}\partial_{\xi}\tilde{c}_{\epsilon}\partial_{t}d
        +\tilde{\bm{u}}_{\epsilon}\cdot(\nabla_{\bm{x}}\tilde{c}_{\epsilon}
        +\epsilon^{-1}\bm{n}\partial_{\xi}\tilde{c}_{\epsilon})
        =-\frac{1}{Pe}(\nabla_{\bm{x}}\cdot\tilde{\bm{j}}_{\epsilon}
        +\epsilon^{-1}\partial_{\xi}\tilde{\bm{j}}_{\epsilon}\cdot\bm{n}),
        \label{in: c}\\
        &\tilde{\bm{j}}_{\epsilon}
        =-\frac{2A K \epsilon D^{+}D^{-}q(\tilde{c}_{\epsilon})}
        {2D^{+}D^{-}(1-\tilde{\phi}_{\epsilon}^{2})^{2}+A K q \epsilon (D^{+}+D^{-})
        +A K q \epsilon (D^{-}-D^{+})\tilde{\phi}_{\epsilon}}
        (\nabla_{\bm{x}}\tilde{c}_{\epsilon}
        +\epsilon^{-1}\partial_{\xi}\tilde{c}_{\epsilon}\bm{n}),\label{in: j}
    \end{align}
\end{subequations}
The leading order term of equation $\eqref{in: c}$ is 
\begin{equation}
    \partial_{\xi}\tilde{c}_{0}\partial_{t}d
    +\bm{\tilde{u}}_{0}\cdot\bm{n}\partial_{\xi}\tilde{c}_{0}
    =-\frac{1}{Pe}\partial_{\xi}\tilde{\bm{j}}_{0}\cdot\bm{n}, 
\end{equation}
with the application of equation $\eqref{eq:dcons}$, we have 
\begin{equation}\label{jump: j0}
    -\partial_{\xi}\tilde{\bm{j}}_{0}\cdot\bm{n}=0 
\end{equation}
which means $\tilde{\bm{j}}_{0}\cdot\bm{n}=\mbox{const}$. 

Integrating the equation $\eqref{jump: j0}$ in $(-\infty,\infty)$ we have 
\begin{equation}\label{eqn: jump of flux}
    [\![\bm{j}_{0}\cdot\bm{n}]\!]=-
    \int_{-\infty}^{\infty}\partial_{\xi}\tilde{\bm{j}}_{0}\mathrm{d}\xi\cdot\bm{n}=0. 
\end{equation}
So the flux of concentration is continuous across the interface. 

For the equation $\eqref{in: j}$, we have 
\begin{equation}
    (2D^{+}D^{-}(1-\tilde{\phi}_{\epsilon}^{2})^{2}+A K q \epsilon (D^{+}+D^{-})
    +A K q \epsilon (D^{-}-D^{+})\tilde{\phi}_{\epsilon})\tilde{\bm{j}}_{\epsilon}
    =-2D^{+}D^{-}A K \epsilon q(\tilde{c}_{\epsilon})
    (\nabla_{\bm{x}}\tilde{c}_{\epsilon}
    +\epsilon^{-1}\partial_{\xi}\tilde{c}_{\epsilon}).\nonumber 
\end{equation}
and the leading order term gives us that 
\begin{equation}
    (1-\tilde{\phi}_{0}^{2})^{2}\tilde{\bm{j}}_{0}\cdot\bm{n}
    =-A K q(\tilde{c}_{0}) \partial_{\xi}\tilde{c}_{0}. 
\end{equation}
Integrating this equation in $(-\infty,\infty)$ we have 
\begin{align}
    AK[\![Q(c_{0})]\!]=
    &AK\int_{-\infty}^{+\infty}q(\tilde{c}_{0})\partial_{\xi}\tilde{c}_{0}\mathrm{d}\xi
    \nonumber\\
    =&\int_{-\infty}^{+\infty}(1-\tilde{\phi}_{0}^{2})^{2}\tilde{\bm{j}}_{0}\cdot\bm{n}\mathrm{d}\xi 
    \nonumber\\
    =&\bm{j}_{0}\cdot\bm{n}\int_{-\infty}^{+\infty}(1-\tanh^{2}\frac{\xi}{\sqrt{2}})^{2}\mathrm{d}\xi. 
\end{align} 
By defining 
\begin{equation}
    A=\int_{-\infty}^{+\infty}(1-\tanh^{2}\frac{\xi}{\sqrt{2}})^{2}\mathrm{d}\xi
    =\frac{4\sqrt{2}}{3}=2\sigma. 
\end{equation} 
we obtain 
\begin{equation}
    \bm{j}_{0}\cdot\bm{n}=K[\![Q(c_{0})]\!]. 
\end{equation}

Combining the above analysis, at the leading order, we obtain the following results: 
\begin{subequations}\label{eq: sharpinterfaceresults}
    \begin{align}
        &Re\big(\frac{\partial\bm{u}_{0}^{\pm}}{\partial t}
        +(\bm{u}_{0}^{\pm}\cdot\nabla)\bm{u}_{0}^{\pm}\big)
        +\nabla p_{0}^{\pm}
        =\nabla^{2}\bm{u}_{0}^{\pm},\quad \mbox{in}~\Omega_{\pm},
        \\
        &\nabla\cdot\bm{u}_{0}^{\pm}=0,\quad \mbox{in}~\Omega_{\pm}, 
        \\
        &\frac{\partial c_{0}^{\pm}}{\partial t}
        +(\bm{u}_{0}^{\pm}\cdot\nabla)c_{0}^{\pm}
        =\frac{1}{Pe}\nabla\cdot (D^\pm \nabla c^\pm),\quad \mbox{in}~\Omega_{\pm}, 
        \\
        &[\![\bm{u}_{0}]\!]=0, \quad \mbox{on}~\Gamma,
        \\
        &-D^+ \nabla c^+\cdot\bm{n}=-D^- \nabla c^-\cdot\bm{n}=K[\![Q(c_{0})]\!], \quad \mbox{on}~\Gamma,
        \\
        &[\![-p_{0}\bm{n}+(\bm{n}\cdot\nabla)\bm{u}_{0}]\!]=
        \frac{1}{Ca}\sigma\kappa\bm{n}, \quad \mbox{on}~\Gamma,
        \\
        &V_{n}=\bm{u}_{0}\cdot\bm{n},\quad \mbox{on}~\Gamma. 
    \end{align}
\end{subequations}
This illustrates that our diffusive interface model converges to the sharp interface model for mass transport with constrict diffusion condition \eqref{eqn: flux_c}. 

\section{Numerical Method}\label{sec: scheme}

The system $\eqref{eqn: dimensionaless system}$ is a totally nonlinearly coupled model. 
In this section, we focus on developing an unconditionally energy stable numerical scheme 
for our proposed mathematical model based on the stabilization method
\cite{Xu2006Stability,Shen2010stability,Zhu1999Coarsening,Shen2015Stability,Yang2010Stability}.

In order to assure the conservation of volume of the numerical scheme, we need to rewrite some terms of system $\eqref{eqn: dimensionaless system}$, 
taking into account the following relations: 
\begin{equation}
    \left\{
        \begin{aligned}
            &\mu\nabla\phi=\nabla(\phi\mu)-\phi\nabla\mu,\\
            &\bm{u}\cdot\nabla\phi=\nabla\cdot(\phi\bm{u}).
        \end{aligned}
    \right.
\end{equation}
Then, we redefine the pressure therm as 
\begin{equation}
    \tilde{p}=p-\frac{1}{Ca}\phi\mu.
\end{equation}
For simplicity of notation, we rewrite $p$ instead of $\tilde{p}$. 

Given initial conditions $\phi^{0}$, $\mu^{0}$, $\bm{u}^{0}$, $p^{0}$ and $c^{0}$, 
we compute ($\phi^{n+1}$, $\mu^{n+1}$, $\tilde{\bm{u}}^{n+1}$, $\bm{u}^{n+1}$, $p^{n+1}$, $c^{n+1}$) 
for $n\geq 0$ by the following steps.

{\bf Step 1.} We solve phase field variable $\phi^{n+1}$ by the following scheme with the help of 
Stabilization method and Navier-Stokes equation by pressure correction method
\cite{Chorin1968NS,Shen2006Projection}: 
\begin{subequations}\label{num: step1}
    \begin{align}
        &Re\big(\frac{\tilde{\bm{u}}^{n+1}-\bm{u}^{n}}{\delta t}
        +(\bm{u}^{n}\cdot\nabla)\tilde{\bm{u}}^{n+1}\big)+\nabla p^{n}
        =\Delta\tilde{\bm{u}}^{n+1}
        -\frac{1}{Ca}\phi^{n}\nabla\mu^{n+1},\label{num: tu}
        \\
        &\frac{\phi^{n+1}-\phi^{n}}{\delta t}
        +\nabla\cdot(\tilde{\bm{u}}^{n+1}\phi^{n})
        =\nabla\cdot(\mathcal{M}\nabla\mu^{n+1}),
        \label{num: phi}\\
        &\mu^{n+1}=-\epsilon\Delta\phi^{n+1}
        +\frac{s}{\epsilon}(\phi^{n+1}-\phi^{n})
        +\frac{1}{\epsilon}G^{\prime}(\phi^{n}), 
        \label{num: mu_phi}
    \end{align}
\end{subequations}
with the boundary condition 
\begin{equation}
    \nabla\phi^{n+1}\cdot\bm{n}|_{\partial\Omega}=0,
    \quad 
    \nabla\mu^{n+1}\cdot\bm{n}|_{\partial\Omega}=0, 
    \quad 
    \tilde{\bm{u}}^{n+1}|_{\partial\Omega}=0. 
\end{equation}

{\bf Step 2.} Projection step: 
\begin{subequations}\label{num: step2}
    \begin{align}
        &Re\frac{\bm{u}^{n+1}-\tilde{\bm{u}}^{n+1}}{\delta t}+\nabla(p^{n+1}-p^{n})=0,
        \label{num: u}\\
        &\nabla\cdot\bm{u}^{n+1}=0, \label{num: nabla_u}
    \end{align}
\end{subequations}
with boundary condition 
\begin{equation}
    \bm{u}^{n+1}|_{\partial\Omega}=0. 
\end{equation}

{\bf Step 3.} $c^{n+1}$ will be updated by the following fully implicit Euler scheme: 
\begin{subequations}\label{num: step4}
    \begin{align}
        &\frac{c^{n+1}-c^{n}}{\delta t}+(\bm{u}^{n+1}\cdot\nabla)c^{n+1}
        =-\nabla\cdot\bm{j}^{n+1},
        \label{num: c}\\
        &\bm{j}^{n+1}=-D^{n+1}c^{n+1}\nabla\mu_{c}^{n+1}, 
        \label{num: flux}\\
        &\mu_{c}^{n+1}= \ln c^{n+1} ,
        \label{num: mu_c}\\
        &\frac{1}{D^{n+1}}=
        \frac{(1-(\phi^{n+1})^{2})^{2}}{AK q(C^{n}) \epsilon}
        +\frac{1-\phi^{n+1}}{2D^{-}}+\frac{1+\phi^{n+1}}{2D^{+}}, 
    \end{align}
\end{subequations}
with boundary condtion 
\begin{equation}\label{num: c_bdc}
    \nabla c^{n+1}\cdot\bm{n}|_{\partial\Omega}=0,
    \quad 
    \nabla\mu_{c}^{n+1}\cdot\bm{n}|_{\partial\Omega}=0. 
\end{equation}

\begin{theorem}
    System $\eqref{num: step1}-\eqref{num: c_bdc}$ is uniquely solvable, 
    unconditionally stable and obey the following discrete energy law: 
    \begin{align}
        &\frac{Re}{2}(\|\bm{u}^{n+1}\|^{2}
        -\|\bm{u}^{n}\|^{2})
        +\frac{\delta t^{2}}{2Re}(\|\nabla p^{n+1}\|^{2}-\|\nabla p^{n}\|^{2})
        \nonumber\\
        &+\frac{1}{Ca}\big(\frac{\epsilon}{2}(\|\nabla\phi^{n+1}\|^{2}-\|\nabla\phi^{n}\|^{2})
        +\frac{1}{\epsilon}(G(\phi^{n+1})-G(\phi^{n}),1)\big)
        \nonumber\\
        &+(c^{n+1}\ln c^{n+1}-c^{n}\ln c^{n},1)
        \nonumber\\
        =&-\delta t\|\nabla\tilde{\bm{u}}^{n+1}\|^{2}
        -\delta t\frac{\mathcal{M}}{Ca}\|\nabla\mu^{n+1}\|^{2}
        -\delta t\int_{\Omega}D^{n+1}c^{n+1}|\nabla\mu_{c}^{n+1}|^{2}\mathrm{d}\bm{x}, 
    \end{align}
    where $\|\cdot\|$ denotes the discrete $L^{2}$ norm in domain $\Omega$. 
\end{theorem}

\begin{remark}
    Numerical scheme $\eqref{num: step1}$ is a coupled system for $\phi^{n+1}$ and 
    $\tilde{\bm{u}}^{n+1}$. Block Gauss iteration method is used to compute this system efficiently. 
\end{remark}

\section{Numerical results}\label{sec: results}
In this section,   some numerical experiments are conducted to illustrate  validity of our model.
We first check the sharp interface limit results in Section \ref{sec: sharp} by choosing smaller $\epsilon$ and comparing with analytical solutions of  sharp interface models. Then we check the convergence rate and energy stability of the numerical scheme in Section \ref{sec: scheme}. Finally, the calibrated model and scheme are used to study the effect of interface permeability. 

Block-centered finite difference method based on stagger mesh is adopted  to discretize equations 
$\eqref{eqn: dimensionaless system}$ in space. 
Variables $\phi$, $c$ and $p$ are located in the center of mesh, 
however velocity variables $u$ and $v$ are located on the center of edge. 
The main advantage of the block-centered finite difference method is 
it approximates the phase function, concentration function and pressure function with 
Neumann boundary condition to second-order accuracy, and also, it guarantees local mass conservation.

\subsection{Sharp interface limit test}
In this example, we take the steady state in 1D case to verify sharp interface limit of 
concentration function $c$. For simplicity, we first fixed the interface and only solve concentration equation \eqref{eqn: c} where interface is assumed at $x_{0}=0.5$ in domain $(0,1)$.
The diffusion coefficient is taken as $D_{1}=D_{2}=1$ and 
Dirichlet boundary condition $c_{0}=1,~c_{1}=4$ is used. 
Then in this case,  the exact solution of sharp interface model \eqref{eqn: diffusion}-\eqref{eqn: flux law}  is 
\begin{equation}
    c=\left\{
    \begin{aligned}
        &x+1, && x<x_{0}, \\
        &x+3, && x\geq x_{0}.
    \end{aligned}
    \right.
\end{equation}

In Figure $\ref{fig: sharp interface}$, the exact solution of is shown in black solid line and the dash lines are the solutions of Eq. \eqref{eqn: c} where the phase field function is chosen as 
$\phi=\tanh(\frac{x-x_{0}}{\sqrt{2}\epsilon})$ with different interface thickness $\epsilon$. In the bulk region, solutions of two methods fit very well. 
 As $\epsilon\rightarrow 0$, the proposed diffusive model solutions change  much sharper near the interface and convergence to the sharp interface solution, which is consistent with our analysis in Section \ref{sec: convergence}.

\begin{figure}[htbp]
    \begin{center}
    \includegraphics[width=0.8\textwidth]{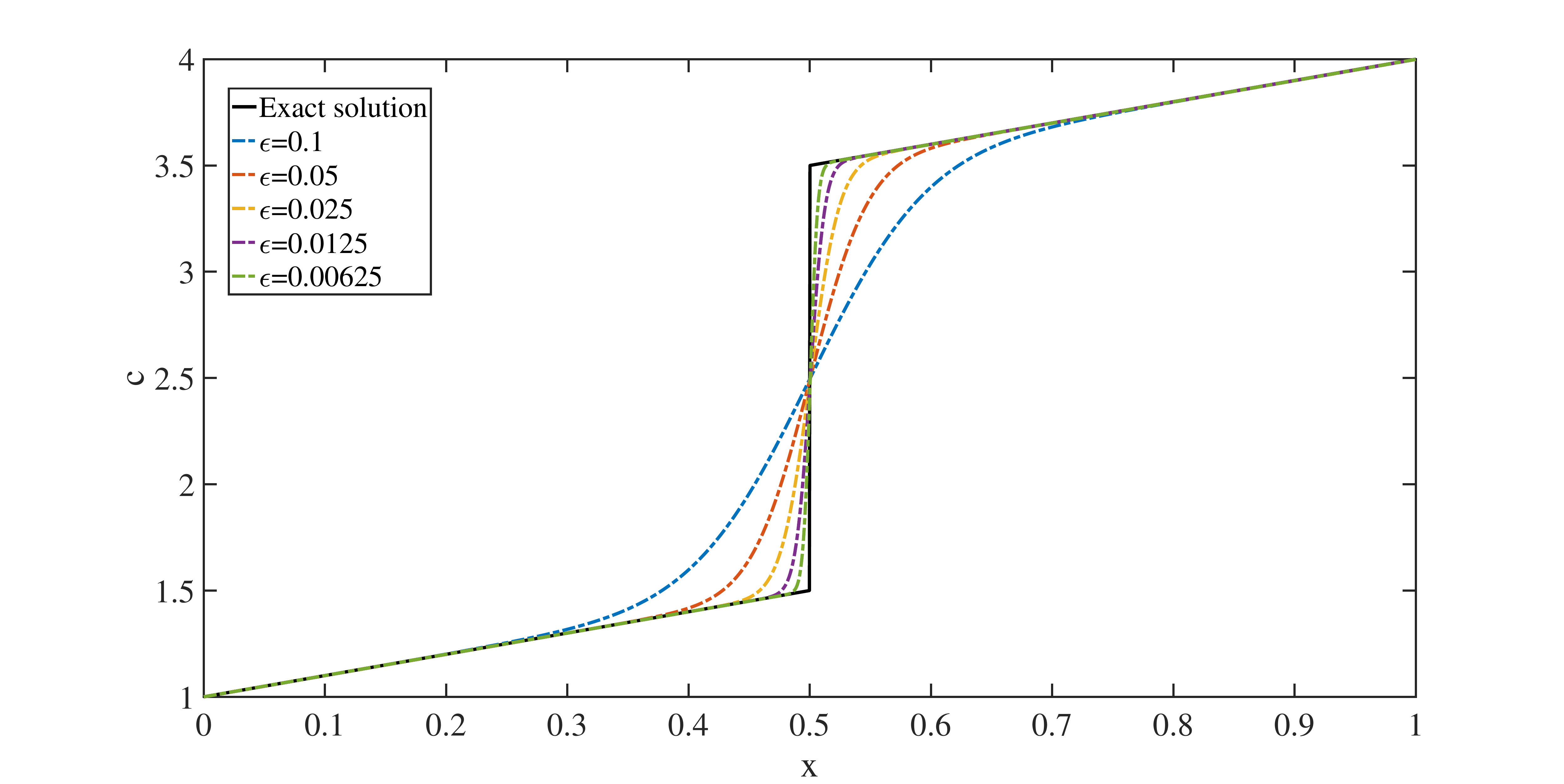}
    \end{center}
    \caption{Sharp interface limit}
    \label{fig: sharp interface}
\end{figure}

\subsection{Comparison with sharp interface model}
In this subsection, we conduct two numerical tests to compare with  the results using immersed boundary methods \cite{Huang2009Immersed}. 
The first experiment is 1D steady state solution of Eqs. \eqref{eqn: diffusion}-\eqref{eqn: flux law}
  with $Q(c)=c$, $K =1/5$.  The   locations of interfaces are chosen  at $x_{1}=7/18$ and $x_{2}=11/18$, i.e. $\Omega^+ = (x_1, x_2)$ and $\Omega^- = (0,x_1)\cup (x_2, 1)$ and diffusion constants are  $D^+ = D^-=1$. In this case, the exact solution can be obtained as 
\begin{equation}
    c=\left\{
    \begin{aligned}
        &-\frac{1}{11}x+2, && x<x_{1}, \\
        &-\frac{1}{11}x+\frac{17}{11}, && x_{1}\leq x<x_{2},\\
        &-\frac{1}{11}(x-1)+1, && x\geq x_{2}.
    \end{aligned}
    \right.
\end{equation}

In Figure $\ref{fig: Huang_1d}$, the solid black line is above exact solution, blue line with circle is the immerse boundary method solution and the red line with square is the the diffusive interface solution with interface width $\epsilon=0.01$.
  
We can see that the numerical solution is almost coincided with the exact solution and 
there is  indistinguishable difference between our results and the immerse boundary result in \cite{Huang2009Immersed}. 

\begin{figure}[h!]
        \begin{center}
            \includegraphics[width=0.8\textwidth]{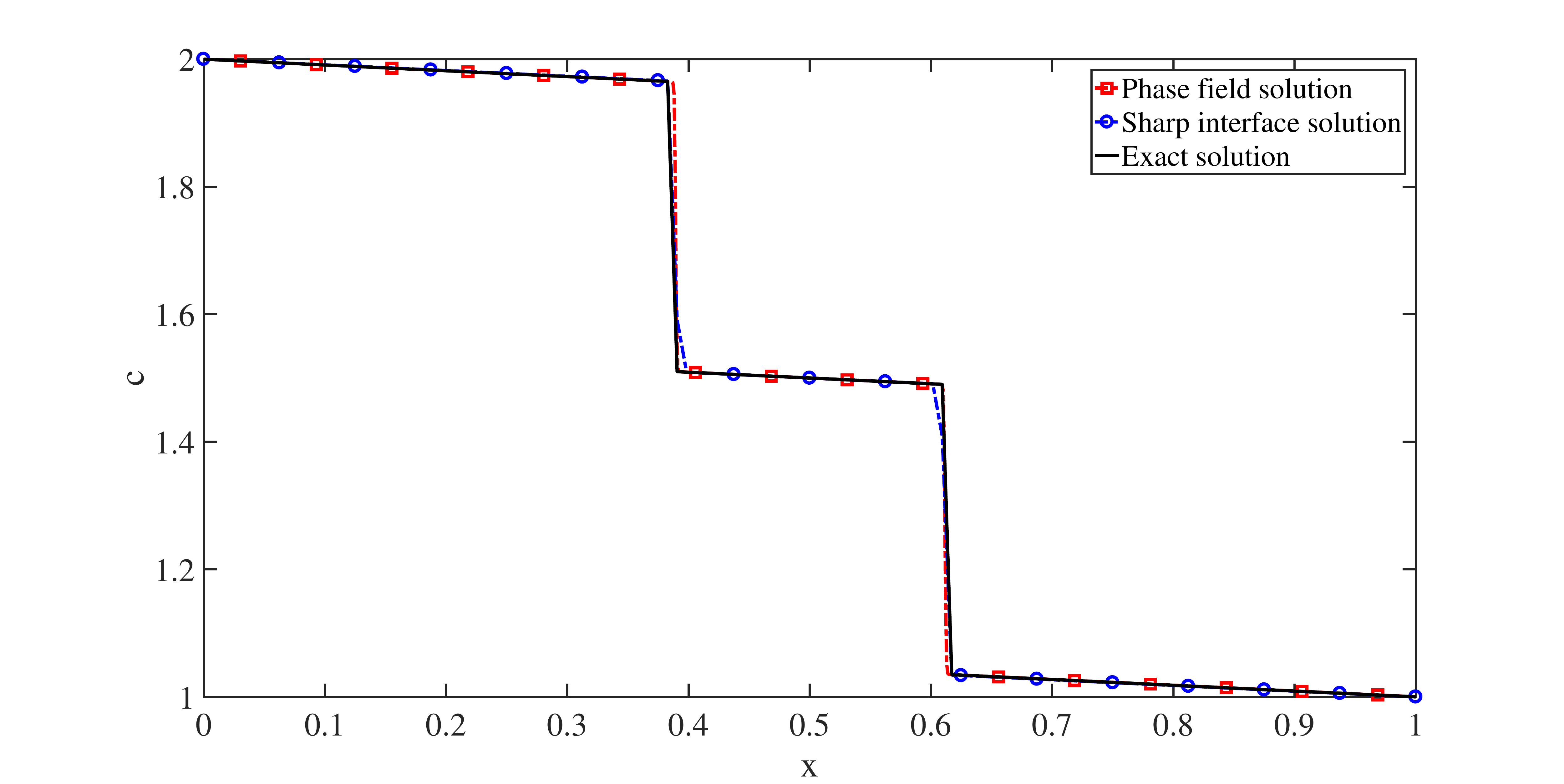}
        \end{center}
    \caption{Steady state solution with the linear law on the interfaces located at $x_{1}=7/18$ 
    and $x_{2}=11/18$ with $K=0.2$, $D=1$, $c_{0}=2$ and $c_{1}=1$. The black solid line is the exact solution, 
    and the red dash with square is the numerical solution by phase field method, 
    but the blue dash with circle is the numerical solution by immersed boundary method.} 
    \label{fig: Huang_1d}
\end{figure}

In the next, the 2D  validation is conducted where the interface is fixed as a circle with radius $r=11/18$  and center at $\bm{x}_0=(0.5, 0.5)$. 
The  initial  condition of concentration (see Fig. \ref{fig: Huang} (a)) and parameters 
are set to be same as those in \cite{Huang2009Immersed}, 
 $c(t_{0},\bm{x})=G(t_{0},\bm{x}-\bm{x}_{0})$, where
\begin{equation}
    G(t,\bm{x}-\bm{x}_{0})=\frac{1}{4\pi Dt}\exp{(-\frac{\bm{x}-\bm{x}_{0}}{4Dt})}.
\end{equation}
with 
\begin{equation}
    t_{0}=10^{-4},\quad 
    D = 1,\quad
    K = 10.
\end{equation}
The computational domain is discretized by a uniform grid with size $1/128$ and interface thickness $\epsilon = 0.01$.  

Fig.\ref{fig: Huang} (b) show the distribution of concentrations  by using immerse boundary method (dash lines) \cite{Huang2009Immersed} and proposed diffusive interface method (solid lines) at time $t = t_0+10^{-2}$.
 It illustrates that our diffusive interface model fits the immerse interface results very well for the concentration through the semi-permeable interface with restrict diffusion.


\begin{figure}[h!]
    \subfigure[Initial condition of concentration.]
    {
    \begin{minipage}[h]{0.5\linewidth}
        \begin{center}
            \includegraphics[width=0.8\textwidth]{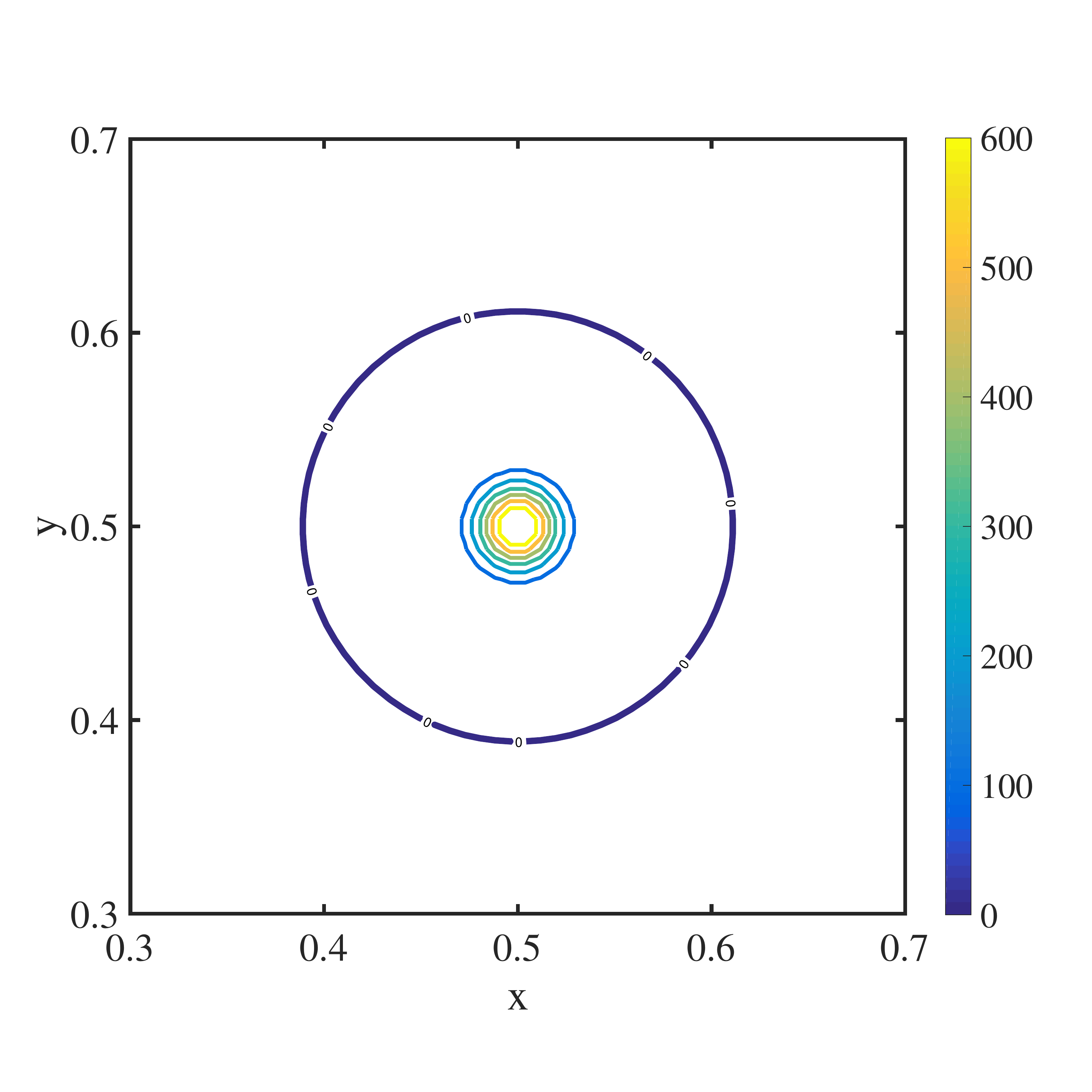}
        \end{center}
    \end{minipage} 
    }
    \subfigure[Snapshot of level curves of concentration at $t=t_{0}+10^{-2}$.]
    {
    \begin{minipage}[h]{0.5\linewidth}
        \begin{center}
            \includegraphics[width=0.8\textwidth]{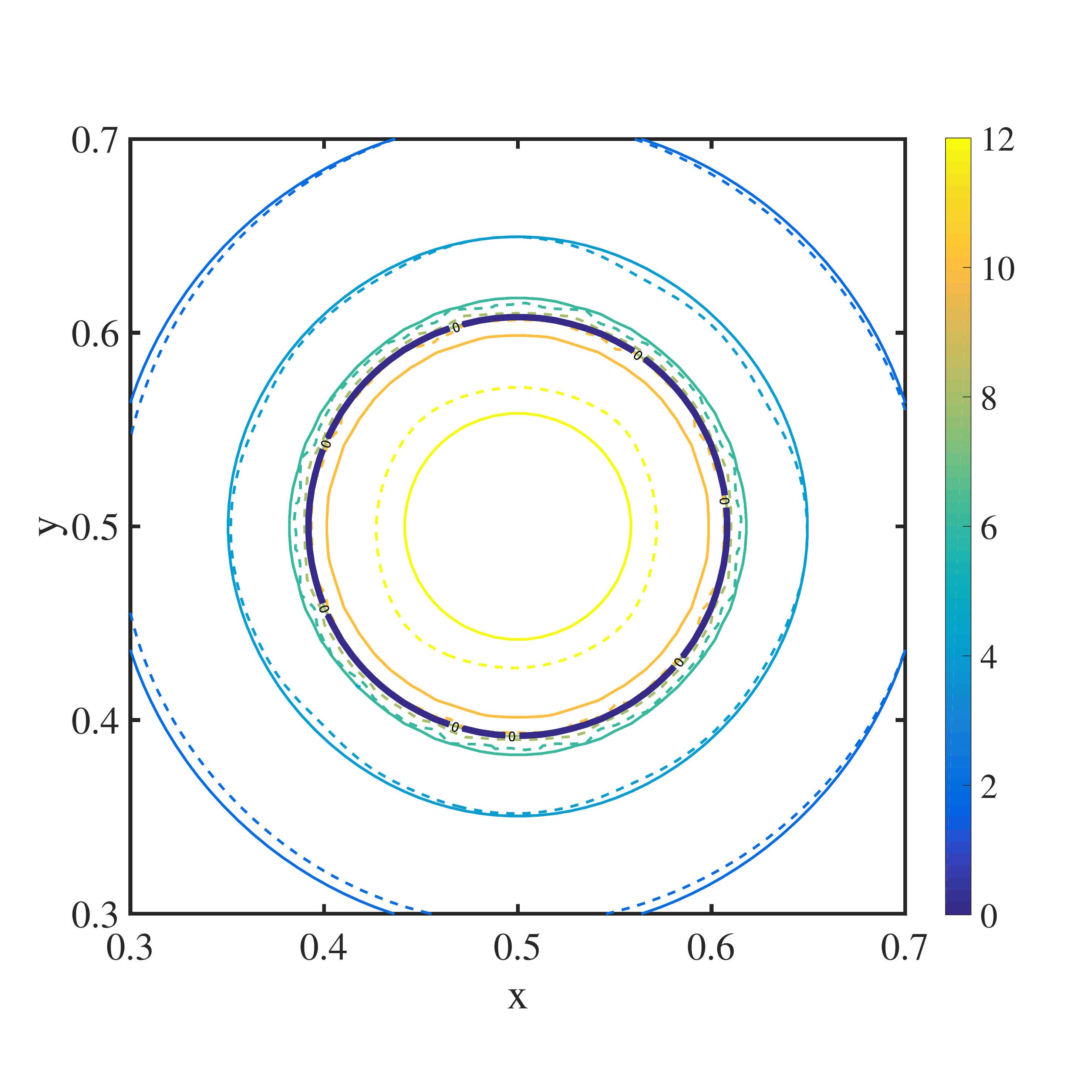}
        \end{center}
    \end{minipage} 
    }
    \caption{Comparison between phase field method and immersed boundary method for mass transfer. The fixed interface is expressed with bold cure marked by 0. 
   (a): The initial condition of concentration that concentrates in the centre. (b): Snapshot of level curves of concentration $t=t_{0}+10^{-2}$ by using different methods.
 The dashed lines are  level curves obtained by immersed boundary method \cite{Huang2009Immersed}  and the solid lines are the level curves obtained by  proposed diffusive interface method. } 
    \label{fig: Huang}
\end{figure}

\subsection{Convergence study and unconditional energy stability}
In this section, we conduct the convergence test and unconditional energy stable test to 
illustrate the effectiveness of our proposed numerical scheme. 
Here we use the uniform mesh, namely we use the mesh size $N_{x}=N_{y}=N$, and if not specified, 
uniform mesh is always tenable. 
\subsubsection{Convergence rate test}\label{sec: convergence}
In this subsection, we perform some numerical experiments to support the theoretical results. 
We use a uniform Cartesian grid to discretize a square domain $\Omega=(0,1)^{2}\subset \mathbb{R}^{2}$. 
     The initial condition is chosen as follows,  (\textcolor{red}{ see Fig. \ref{fig: initial convergence}})
\begin{subequations}\label{eqn: convergence}
    \begin{align}
        &\phi(\bm{x},0)=0.2+0.5\cos(2\pi x)\cos(2\pi y),\\
        &c(\bm{x},0)=0.6+0.2\cos(2\pi x)\cos(2\pi y),\\
        &u(\bm{x},0)=-0.25\sin^{2}(\pi x)\cos(2\pi y),\\
        &v(\bm{x},0)=0.25\sin^{2}(\pi y)\cos(2\pi x).
    \end{align}
\end{subequations}
Periodic boundary conditions are adopted for all variables. 
The parameters in model $\eqref{eqn: dimensionaless system}$ are set as 
\begin{equation}
    \epsilon=0.08,\quad Re = 1,\quad Ca = 1,\quad Pe = 1,\quad D^{+}=D^{-}=1,\quad 
    K=A^{-1},\quad Q(c) = c.
\end{equation}

\begin{figure}[h!]
\begin{center}
    \subfigure[$c(\bm{x},0)$.]{
    \begin{minipage}[h]{0.4\linewidth}
        \begin{center}
            \includegraphics[width=1.0\textwidth]{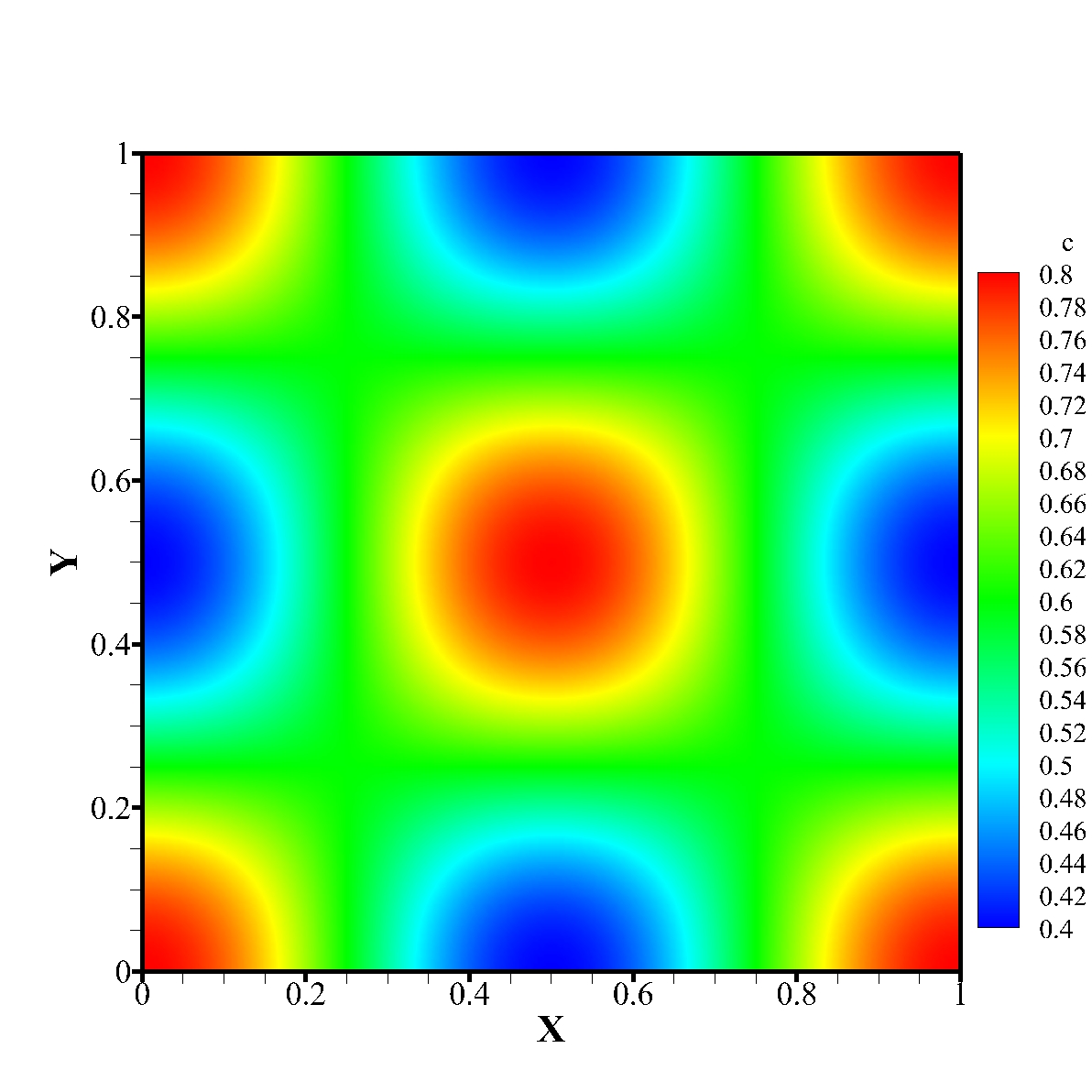}
        \end{center}
    \end{minipage}
    }
    \subfigure[$\phi(\bm{x},0)$.]{
    \begin{minipage}[h]{0.4\linewidth}
        \begin{center}
            \includegraphics[width=1.0\textwidth]{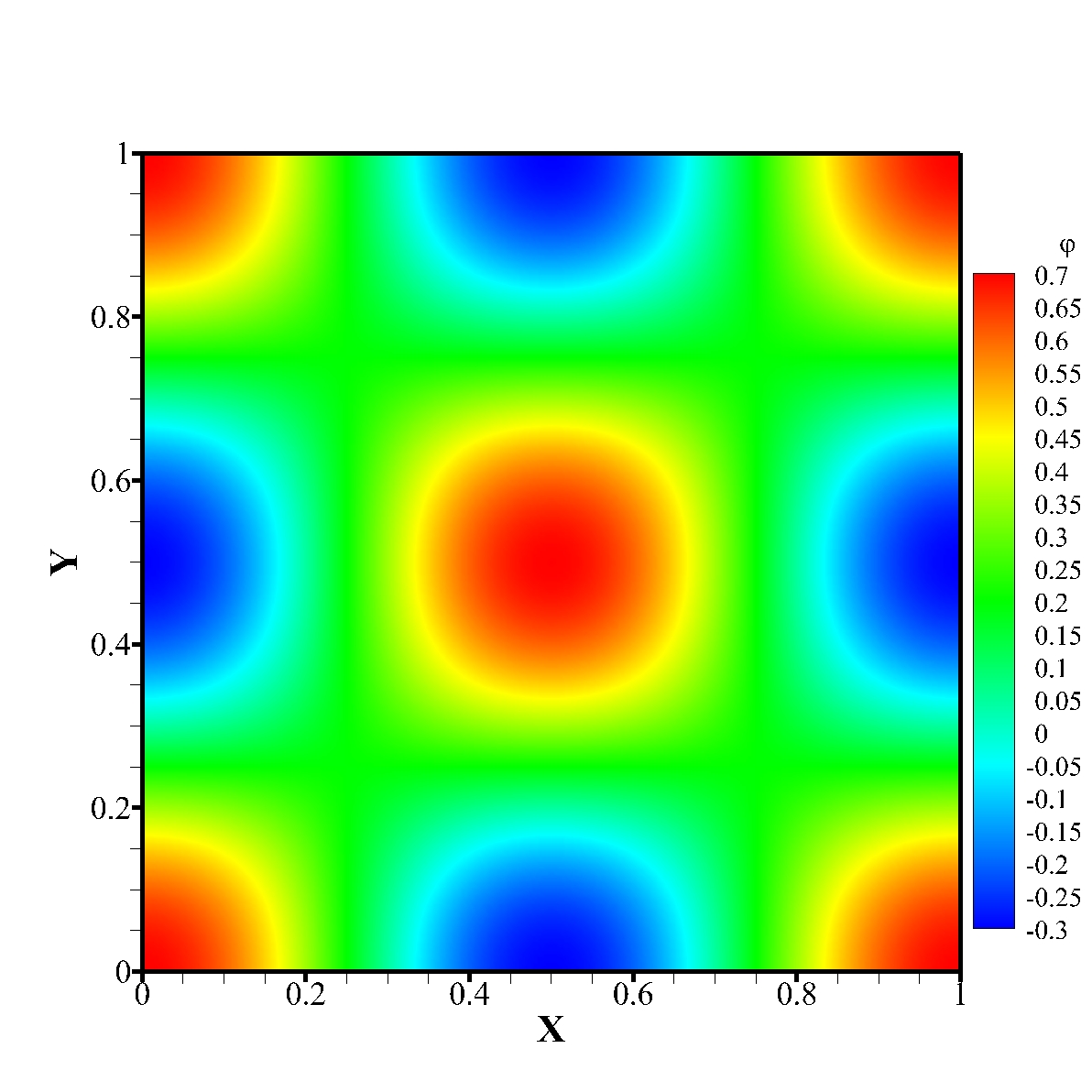}
        \end{center}
    \end{minipage} 
    }
    \end{center}
    \caption{Initial conditions for convergence study. (a): concentration $c(x,0)$; (b): phase field function $\phi(x,0)$.   $\ref{eqn: convergence}$.} 
    \label{fig: initial convergence}
\end{figure}

The Cauchy error \cite{Wise2007Solving} is used to test the convergence rate. 
In this test method, error between two different spacial mesh sizes  $h$ and $h/2$ is calculated by 
$\|e_{\zeta}\|=\|\zeta_{h}-\zeta_{h/2}\|$, where $\zeta$ is the function to be solved. The mesh sizes are set to  be $h=1/16,~1/32, ~1/64, ~ 1/128, ~1/256$ 
 and time step is fixed as $\delta t=10^{-4}$.
The $L^{2}$ and $L^{\infty}$ numerical errors and convergence rate at chosen time $T= 0.1$ are displayed in Table 
$\ref{tab: L2 space convergence}$ and Table $\ref{tab: Linfty space convergence}$, respectively. 
The second order spatial accuracy is apparently observed for all the variables. 

\begin{table}[h!]
    \centering
    \caption{The discrete $L^{2}$ error and convergence rate at $t = 0.1$ 
    with initial data $\eqref{eqn: convergence}$ and the given parameters.}\label{tab: L2 space convergence}
    \vskip 0.2cm
    \begin{tabular}{lcccccccccc}
    \toprule
    Grid sizes    &Error$(\phi)$ & Rate & Error$(c)$ &Rate&Error$(u)$ & Rate & Error$(v)$ &Rate& Error$(p)$ &Rate\\
    \midrule
    $16\times 16$    & 4.01e-02 &  --  & 1.01e-02 &  --  & 1.15e-04 &  --  & 1.15e-04 &  --  & 1.57e-03 &  --  \\
    $32\times 32$    & 8.90e-03 & 2.17 & 4.91e-03 & 1.90 & 3.23e-05 & 1.83 & 3.23e-05 & 1.83 & 1.24e-04 & 3.66 \\
    $64\times 64$    & 2.22e-03 & 2.01 & 6.80e-04 & 1.99 & 8.06e-06 & 2.00 & 8.06e-06 & 2.00 & 3.00e-05 & 2.05 \\
    $128\times 128$  & 5.54e-04 & 2.00 & 1.70e-04 & 2.00 & 2.01e-06 & 2.00 & 2.01e-06 & 2.00 & 7.46e-06 & 2.01 \\
    $256\times 256$  & 1.38e-04 & 2.00 & 4.26e-05 & 2.00 & 5.04e-07 & 2.00 & 5.04e-07 & 2.00 & 1.86e-06 & 2.00 \\
    \bottomrule
    \end{tabular}
\end{table}

\begin{table}[h!]
    \centering
    \caption{The discrete $L^{\infty}$ error and convergence rate at $t = 0.1$ 
    with initial data $\eqref{eqn: convergence}$ and the given parameters.}\label{tab: Linfty space convergence}
    \vskip 0.2cm
    \begin{tabular}{lcccccccccc}
    \toprule
    Grid sizes    &Error$(\phi)$ & Rate & Error$(c)$ &Rate&Error$(u)$ & Rate & Error$(v)$ &Rate& Error$(p)$ &Rate\\
    \midrule
    $16\times 16$    & 7.64e-02 &  --  & 1.89e-02 &  --  & 1.83e-04 &  --  & 1.83e-04 &  --  & 3.99e-03 &  --  \\
    $32\times 32$    & 1.69e-02 & 2.17 & 4.91e-03 & 1.95 & 5.00e-05 & 1.88 & 5.00e-05 & 1.88 & 2.53e-04 & 3.98 \\
    $64\times 64$    & 4.18e-03 & 2.02 & 1.23e-03 & 1.99 & 1.25e-05 & 2.00 & 1.25e-05 & 2.00 & 5.89e-05 & 2.10 \\
    $128\times 128$  & 1.03e-03 & 2.02 & 3.07e-04 & 2.01 & 3.12e-06 & 2.00 & 3.12e-06 & 2.00 & 1.50e-05 & 1.98 \\
    $256\times 256$  & 2.55e-04 & 2.01 & 7.61e-05 & 2.01 & 7.80e-07 & 2.00 & 7.80e-07 & 2.00 & 3.71e-06 & 2.01 \\
    \bottomrule
    \end{tabular}
\end{table}

\subsubsection{Unconditionally energy stable test}
In this subsection, we carry out a numerical experiment to survey the unconditionally energy 
stability about our numerical scheme with the same initial condition and boundary conditions. 
The result is listed in Figure $\ref{fig: stability}$. 
Time steps are chosen from $\delta t=0.1$ to $\delta t=0.1\times \frac 1 {2^{8}}$  with fixed spacial mesh size $h = 1/128$. 

The total energy dynamics over time with different time steps are shown in Fig. \ref{fig: stability}. It illustrates that  the energy is monotonic decrease with different time steps which confirms the unconditional energy stability of the  proposed scheme. 
\begin{figure}[h!]
    \begin{center}
    \includegraphics[width=0.8\textwidth]{ 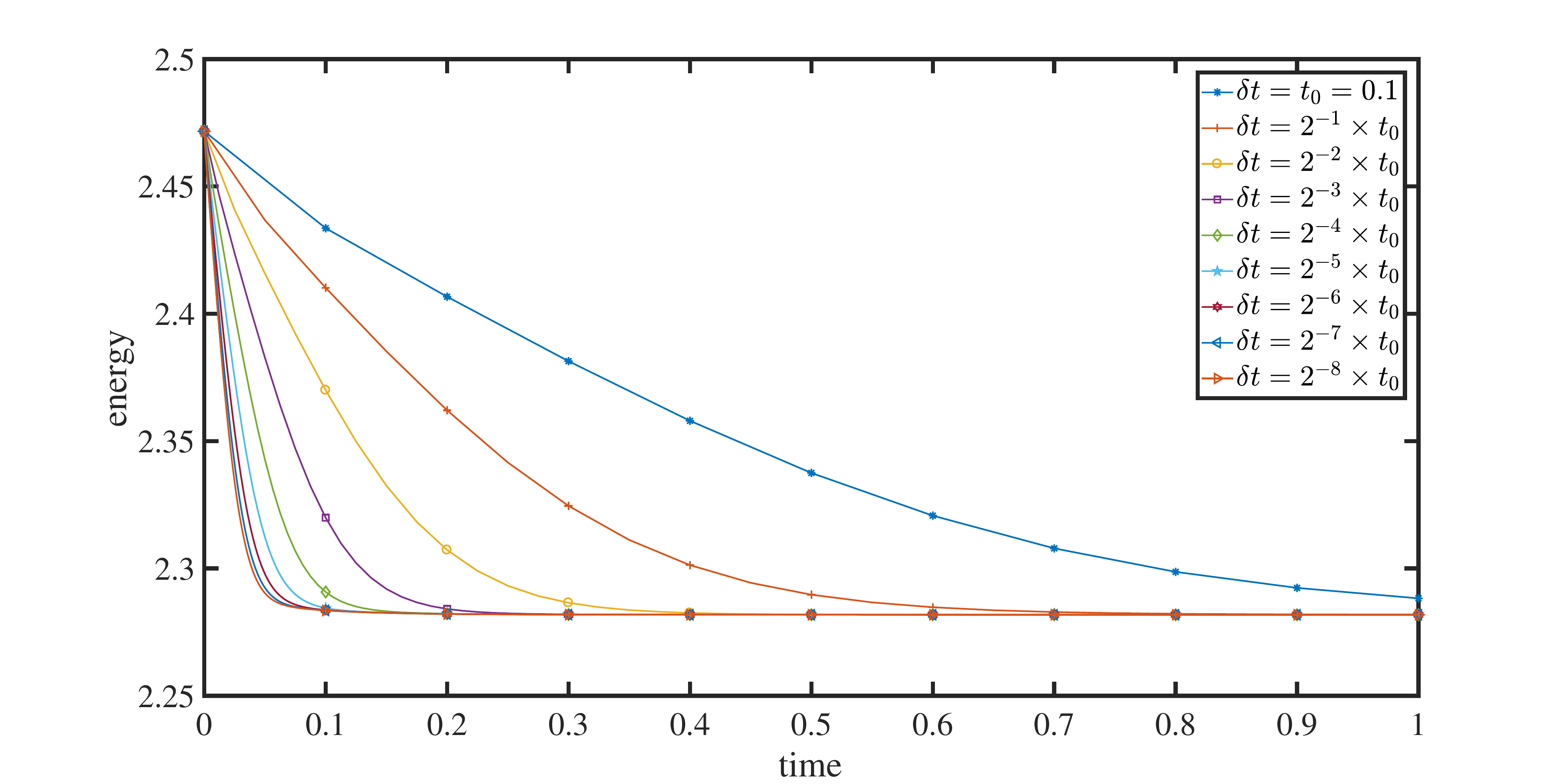}
    \end{center}
    \caption{Unconditional energy stability with different time steps represented by 
    different markers and colors.}
    \label{fig: stability}
\end{figure}

We utilize the result computed with time step $\delta t=0.1\times \frac {1} {2^{8}}$ to give the conservation of volume about our numerical scheme for $\phi$ as illustrated in Fig. $\ref{fig: conservation}$. 
It demonstrates the volume of $\phi$ doesn't change over time. 
\begin{figure}[h!]
    \begin{center}
    \includegraphics[width=0.8\textwidth]{ 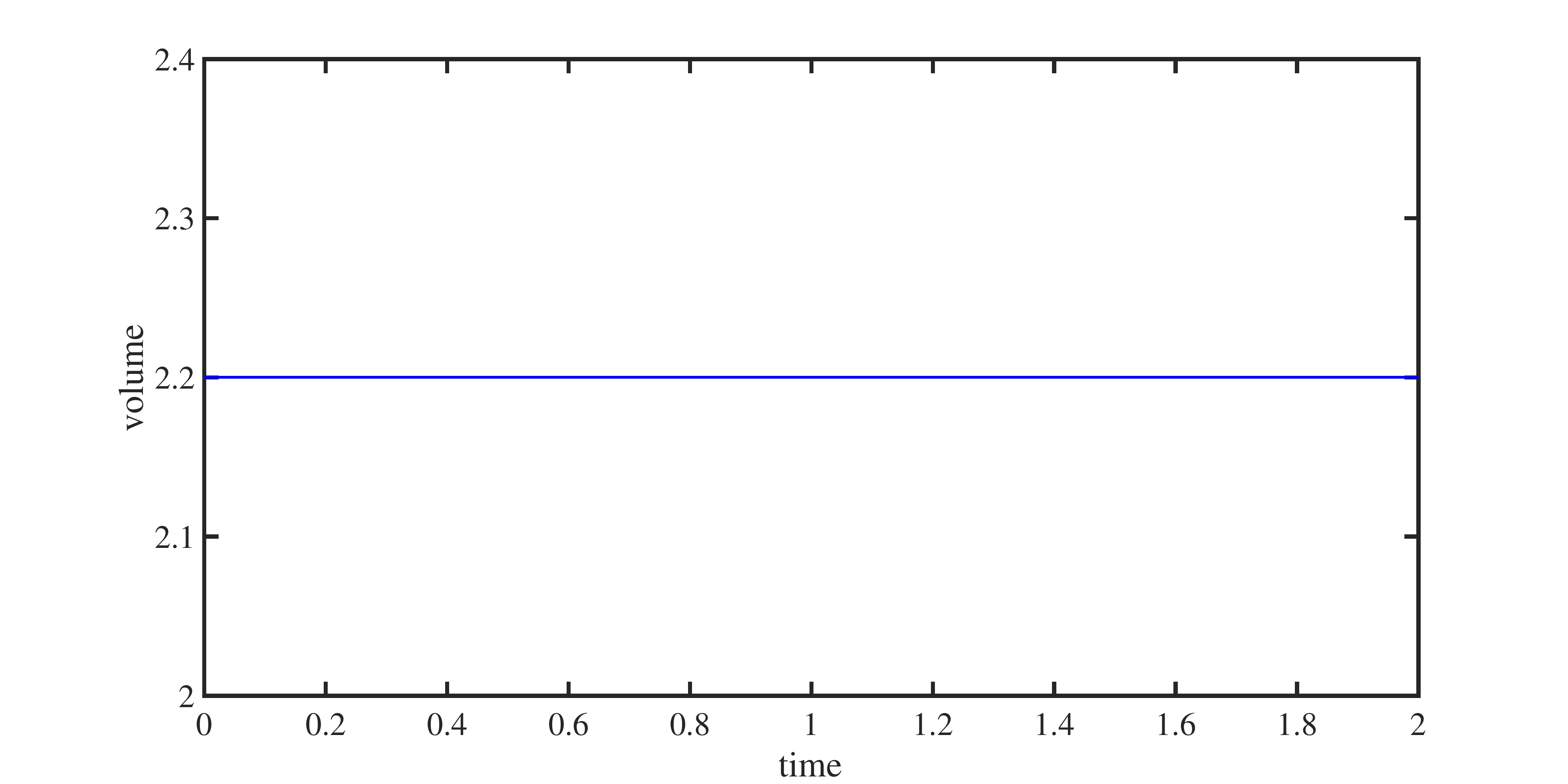}
    \end{center}
    \caption{Conservation of volume of the numerical scheme with time step $\delta t=0.1\times 2^{-8}$.}
    \label{fig: conservation}
\end{figure}

\subsection{Effect of interface permeability}
In this subsection, the calibrated model and scheme are used to study the effect of interface permeability on the diffusion of substance concentration.

We  first consider   a single drop with high ($K = \frac {1}{2\sigma \delta}$), medium  ($K = \frac {1}{2\sigma}  $) and low (($K = \frac {\delta}{2\sigma }$) permeability with $\delta = 0.02$, is suspended in a shear flow.  
The initial  profiles of concentration and interface are given as follows (see Fig. \ref{fig: initial shear}), 
\begin{subequations}\label{exm: comparison}
    \begin{align}
            &\phi({\bm{x},t})=\tanh{\frac{0.25-\sqrt{(x-0.5)^{2}+(y-0.5)^{2}}}{\sqrt{2}\epsilon}},
        \\
        &c(\bm{x},0)=0.6y+0.2. 
    \end{align}
\end{subequations}
with boundary condition 
\begin{align}
  &\nabla\phi\cdot\bm{n}|_{y=0}
  =\nabla\phi\cdot\bm{n}|_{y=1}=0,\quad
  \phi(0,y,t)=\phi(1,y,t),
  \nonumber\\
  &\nabla\mu\cdot\bm{n}|_{y=0}
  =\nabla\mu\cdot\bm{n}|_{y=1}=0,\quad
  \mu(0,y,t)=\mu(1,y,t),
  \nonumber\\
  &\nabla c\cdot\bm{n}|_{y=0}
  =\nabla c\cdot\bm{n}|_{y=1}=0,\quad
  c(0,y,t)=c(1,y,t),
  \nonumber\\
    &
    \bm{u}|_{y=0}=(-1,0)^T,\quad\bm{u}|_{y=1}=(1,0)^T,
    \nonumber
\end{align}
and periodic boundary conditions are used on the left and right boundaries. 
The parameters read  as 
\begin{equation}
Re=100,\quad
Ca=1,\quad
Pe=1,\quad
D^{+}=1,\quad
D^{-}=1,\quad
\epsilon = 0.02,\quad
\mathcal{M}=0.02.
\end{equation}

\begin{figure}[htbp]
    \subfigure[Initial condition of $c$.]{
    \begin{minipage}[h]{0.45\linewidth}
        \begin{center}
            \includegraphics[width=1.0\textwidth]{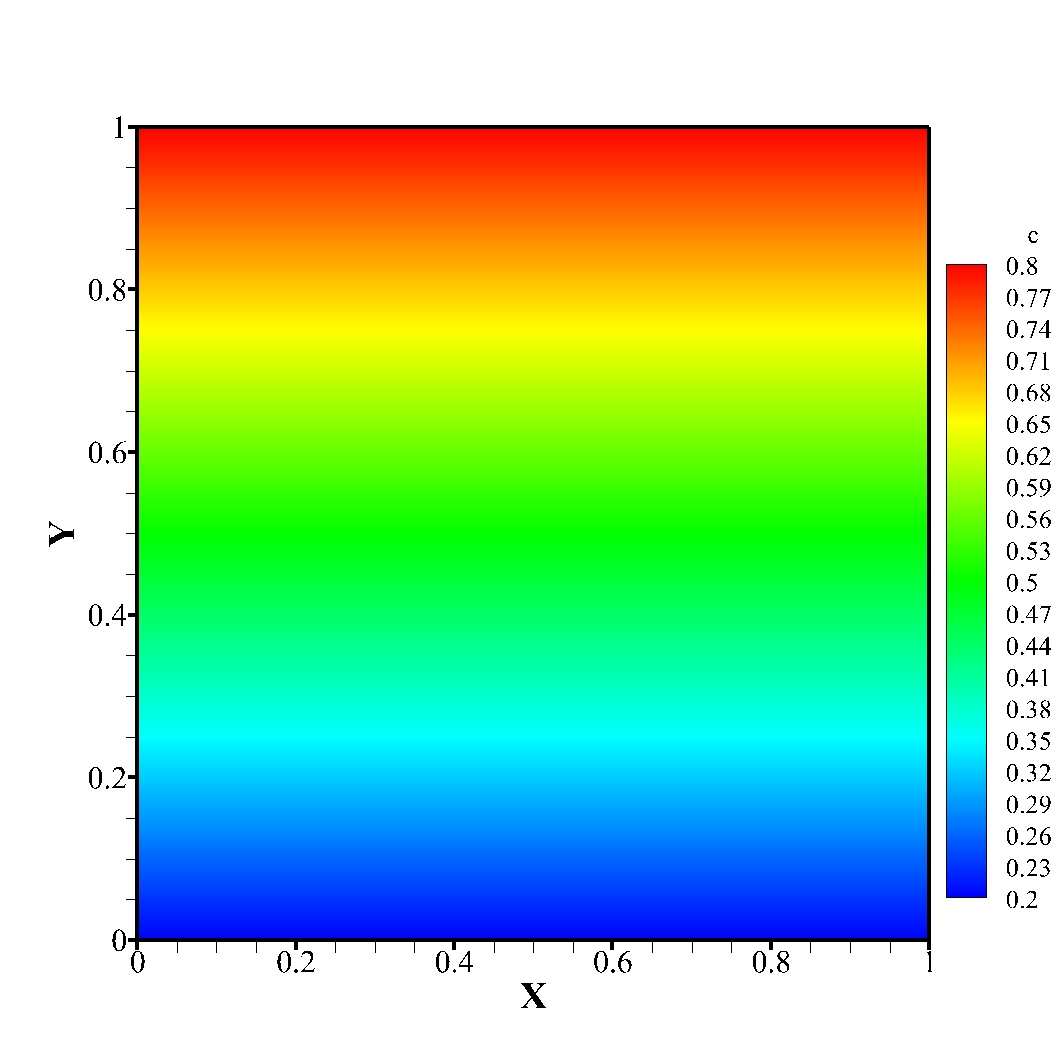}
        \end{center}
    \end{minipage}
    }
    \subfigure[Initial condition of $\phi$.]{
    \begin{minipage}[h]{0.45\linewidth}
        \begin{center}
            \includegraphics[width=1.0\textwidth]{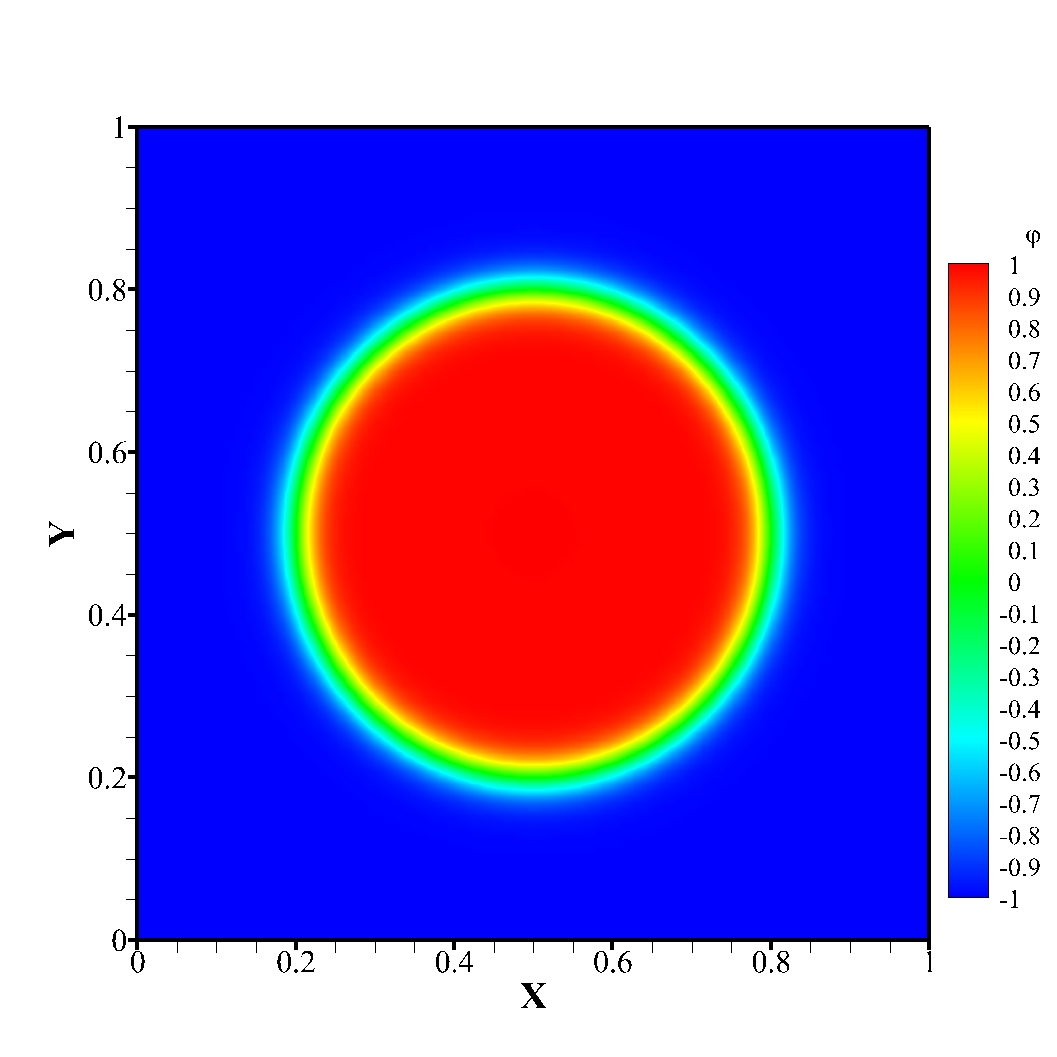}
        \end{center}
    \end{minipage} 
    }
    \caption{Initial condition for example $\ref{exm: comparison}$.
    We set the concentration is a linear function with y-direction.
    The phase field function is a circle with a diffuse interface. 
    The velocity is a shear flow.} 
    \label{fig: initial shear}
\end{figure}

The 2D and 3D profiles of concentration at steady states are shown in Figs. \ref{fig: 2d steady state}-\ref{fig: 3d steady state}.  When the interface with large permeability,  the distribution of concentration is close to the linear solution where no interface is presented. This could also be observed by the direction of flux in Fig.  \ref{fig: flux fielddiagram} (a) and magnitude of flux in Fig. \ref{fig: flux magiagram}(a).  
When the permeability decreases, the flux across the interface decreases (see Fig. \ref{fig: flux magiagram}(b-c)), and are close to nonflux boundary condition such that the straight flux is disturbed around the boundary as shown in Fig. \ref{fig: flux fielddiagram}(b-c) and trans-membrane flux is restricted. 
At the steady state, due to limited flux on the interface, the inner region $\Omega^+$ is almost a constant (see Figs. \ref{fig: 2d steady state}-\ref{fig: 3d steady state} (c)) due to the diffusion inside and very limit flux , while the profile of concentration in the out region $\Omega^-$  is close to the diffusion in a perforated domain. 

\begin{figure}[htbp]
    \subfigure[ Higher permeability 
     ]{
    \begin{minipage}[h]{0.33\linewidth}
        \begin{center}
            \includegraphics[width=1.0\textwidth]{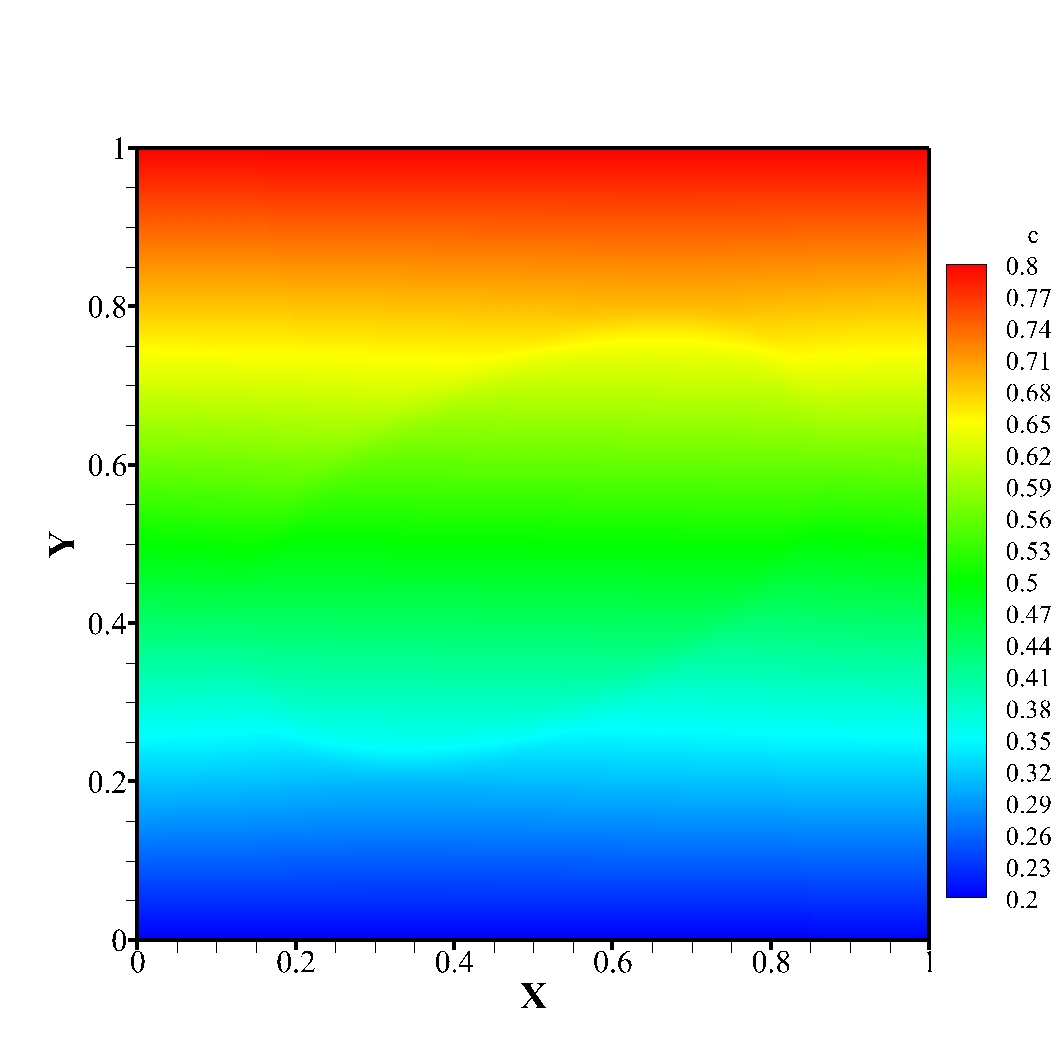} 
        \end{center}
    \end{minipage}
    }
    \subfigure[Medium permeability 
     ]{
    \begin{minipage}[h]{0.33\linewidth}
        \begin{center}
            \includegraphics[width=1.0\textwidth]{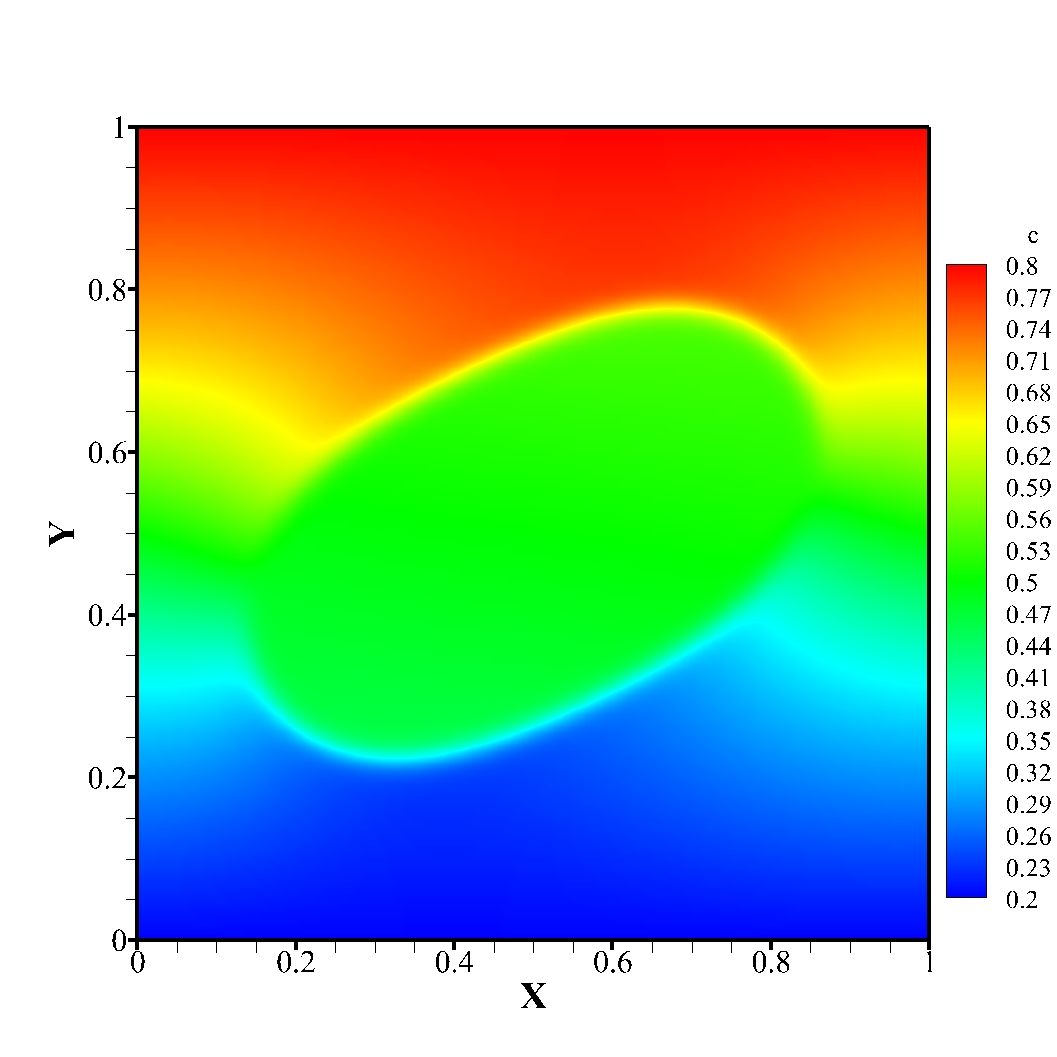} 
        \end{center}
    \end{minipage}
    }
    \subfigure[Lower permeability 
     ]{
    \begin{minipage}[h]{0.33\linewidth}
        \begin{center}
            \includegraphics[width=1.0\textwidth]{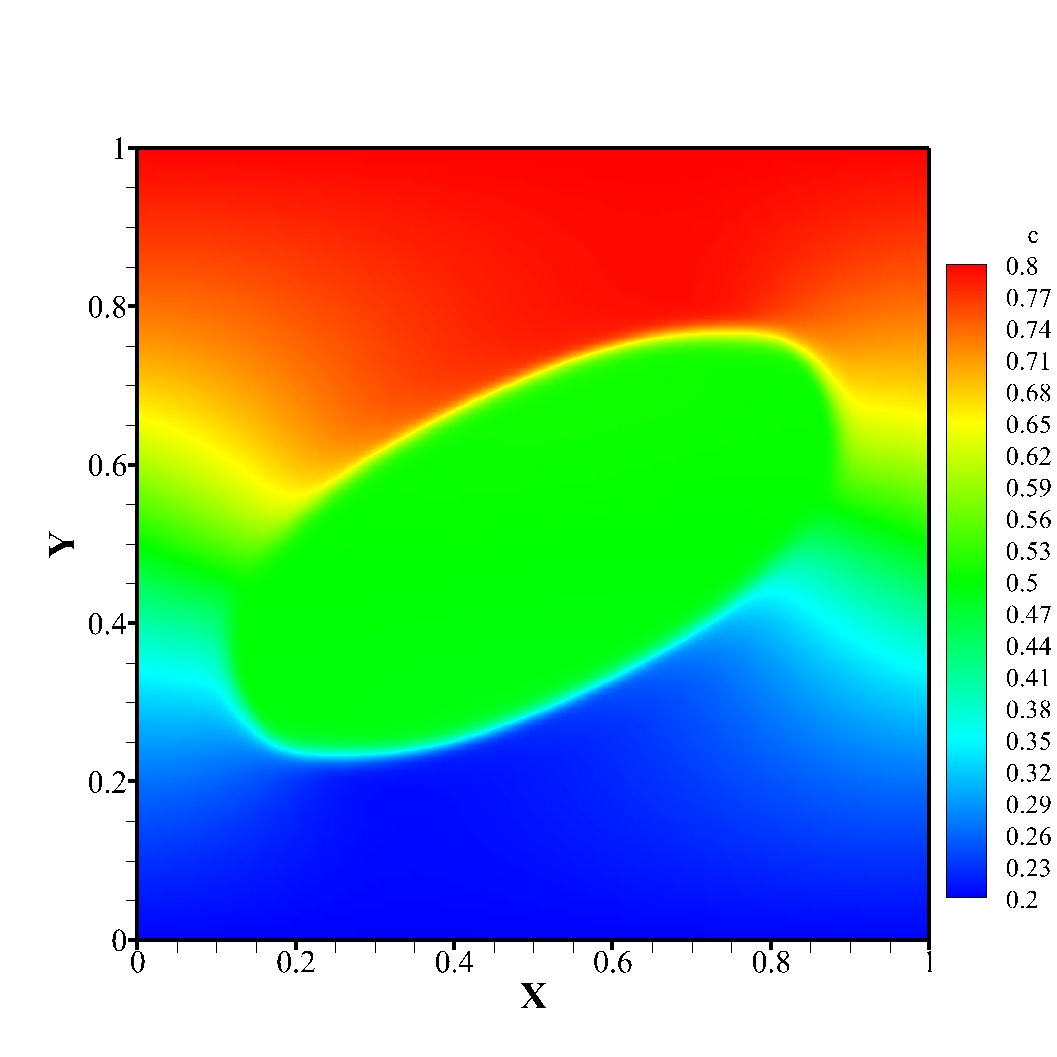} 
        \end{center}
    \end{minipage} 
    }
    \caption{Snapshots of $c$ at steady state with different permeability:  $K=\frac{1}{2\sigma\delta}, \frac{1}{2\sigma}$ and $\frac{\delta}{2\sigma}$.}
    \label{fig: 2d steady state}
\end{figure}

\begin{figure}[htbp]
    \subfigure[Higher permeability 
     ]{
    \begin{minipage}[h]{0.33\linewidth}
        \begin{center}
            \includegraphics[width=1.0\textwidth]{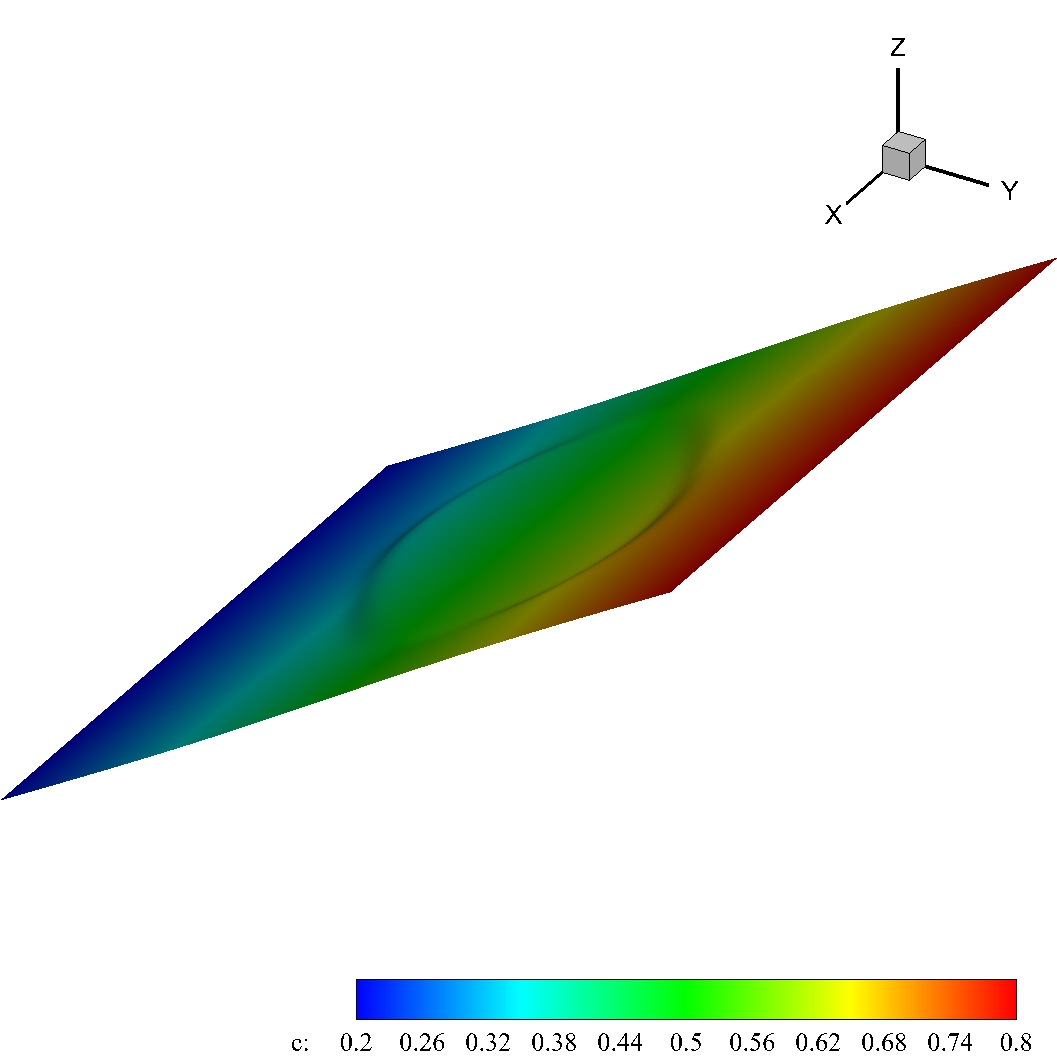} 
        \end{center}
    \end{minipage}
    }
    \subfigure[Medium permeability 
    ]{
    \begin{minipage}[h]{0.33\linewidth}
        \begin{center}
            \includegraphics[width=1.0\textwidth]{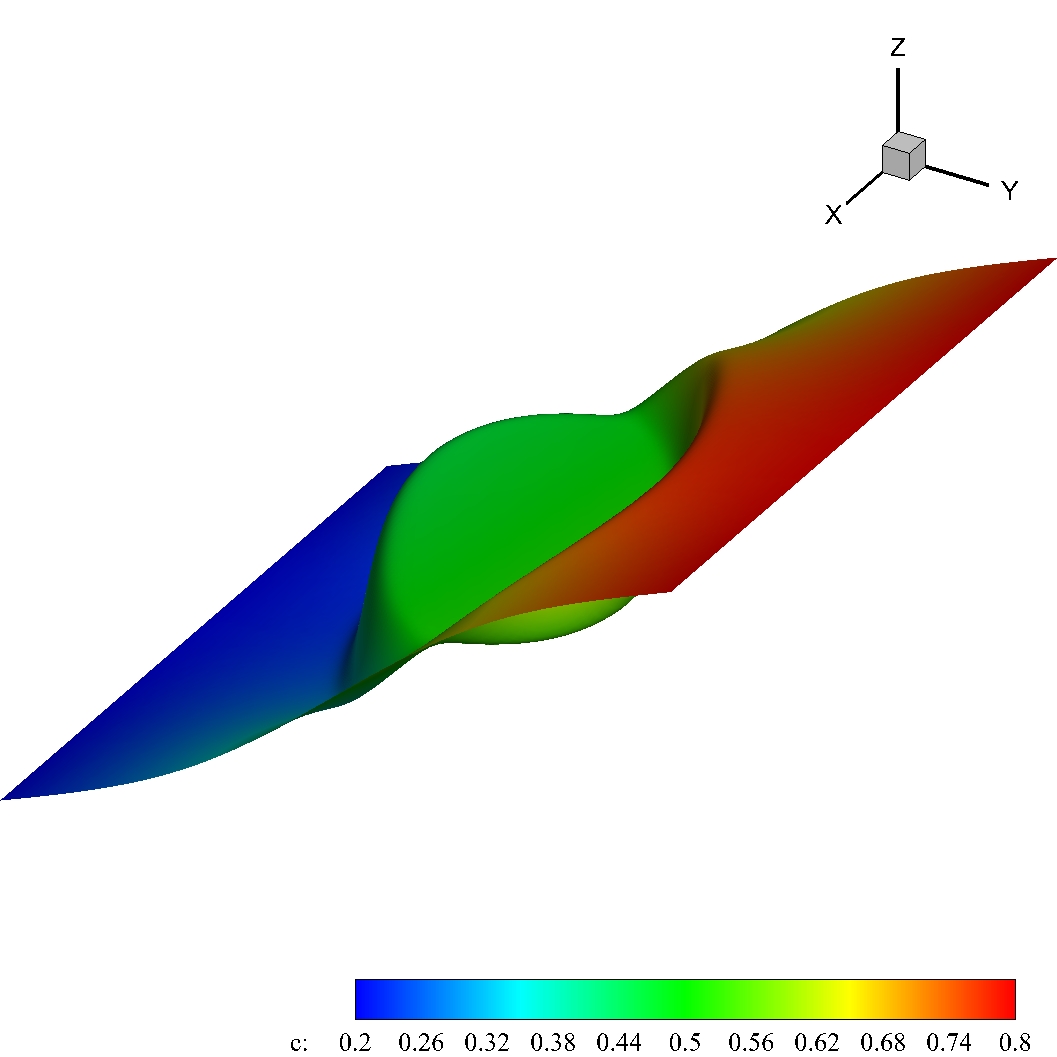} 
        \end{center}
    \end{minipage}
    }
    \subfigure[Lower permeability 
      ]{
    \begin{minipage}[h]{0.33\linewidth}
        \begin{center}
            \includegraphics[width=1.0\textwidth]{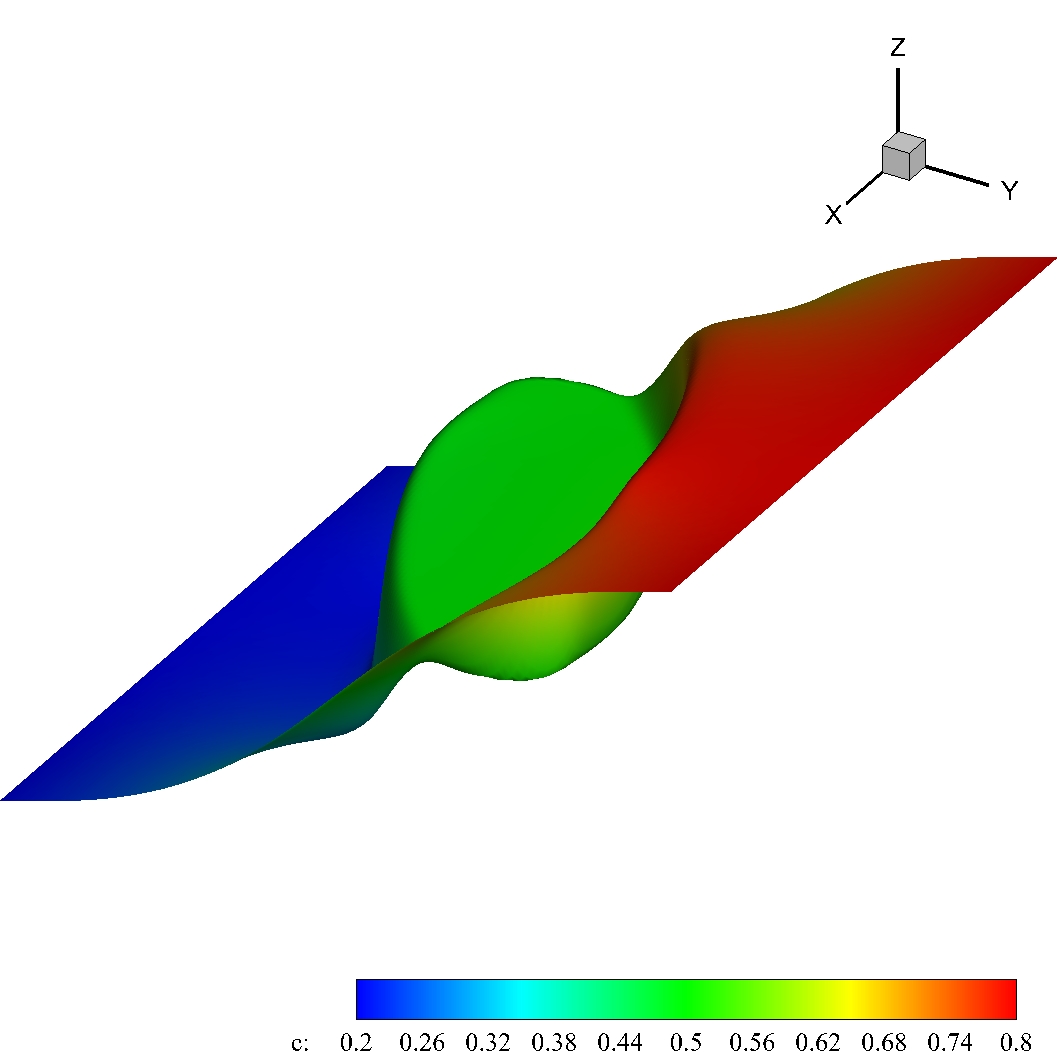} 
        \end{center}
    \end{minipage} 
    }
    \caption{3D profile of $c$ at steady state with different permeabilities 
     $K=\frac{1}{2\sigma\delta}, \frac{1}{2\sigma}$ and $\frac{\delta}{2\sigma}$.}
    \label{fig: 3d steady state}
\end{figure}

\begin{figure}[htbp]
    \subfigure[Higher permeability 
     ]{
    \begin{minipage}[h]{0.33\linewidth}
        \begin{center}
            \includegraphics[width=1.0\textwidth]{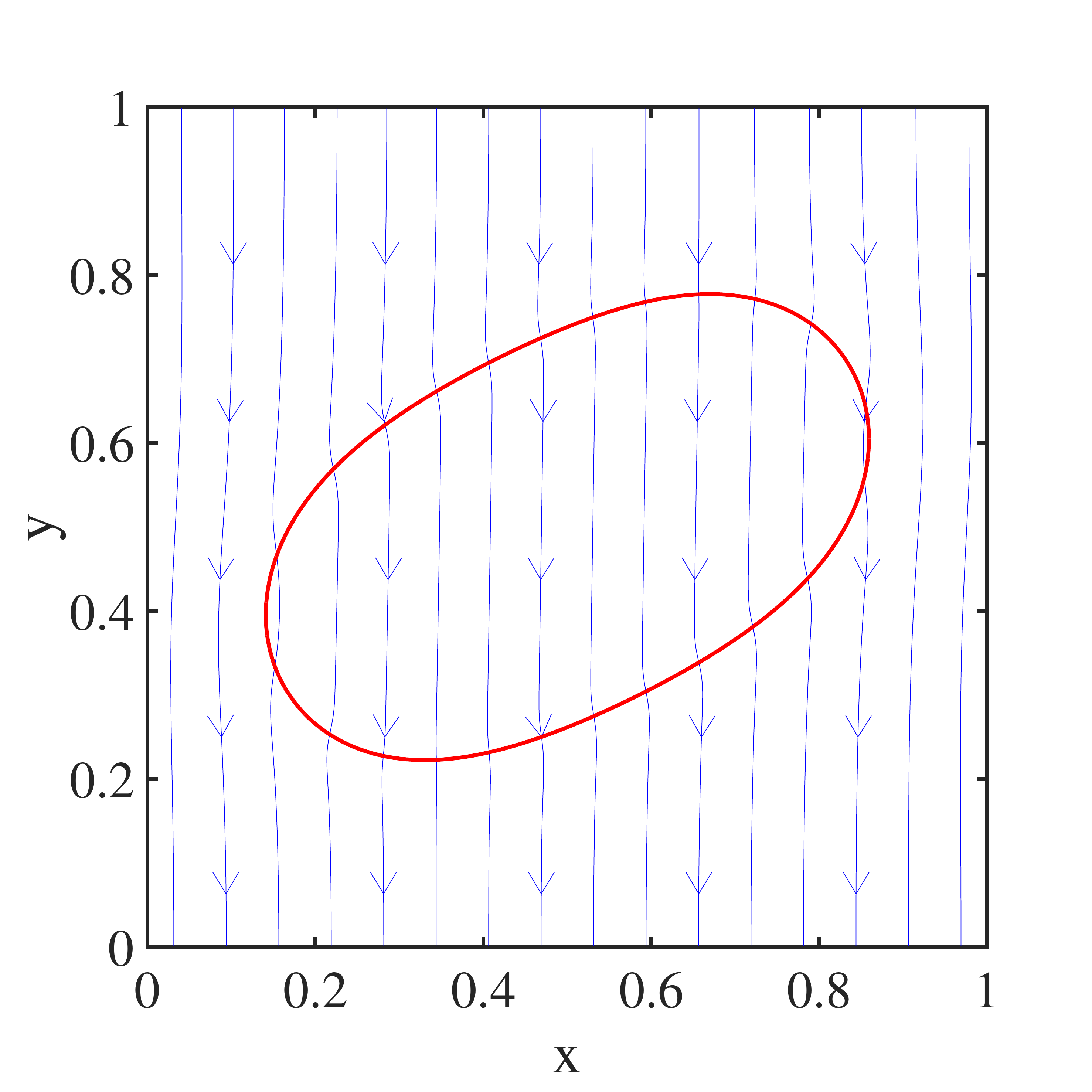} 
        \end{center}
    \end{minipage}
    }
    \subfigure[Medium permeability 
     ]{
    \begin{minipage}[h]{0.33\linewidth}
        \begin{center}
            \includegraphics[width=1.0\textwidth]{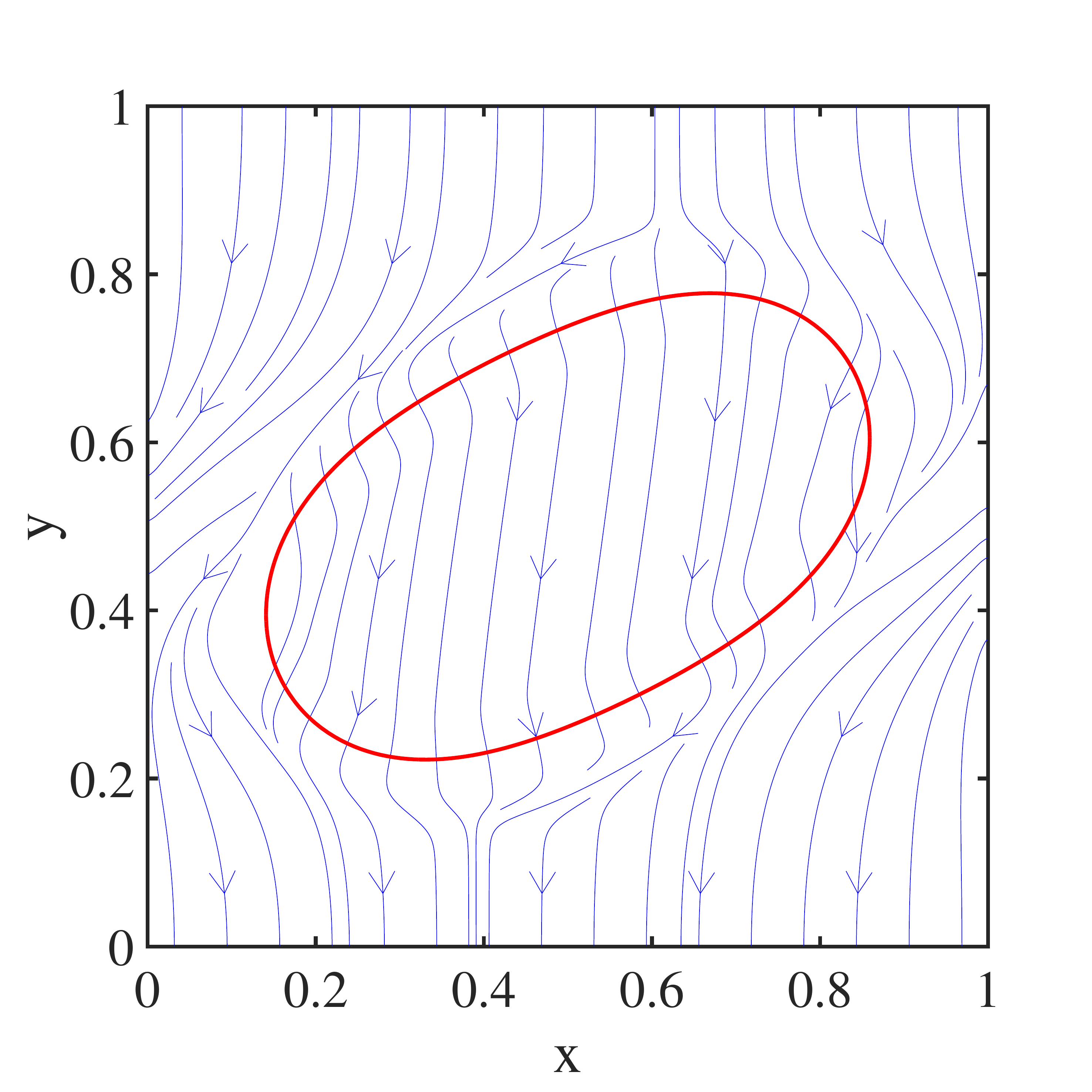} 
        \end{center}
    \end{minipage}
    }
    \subfigure[Lower permeability 
     ]{
    \begin{minipage}[h]{0.33\linewidth}
        \begin{center}
            \includegraphics[width=1.0\textwidth]{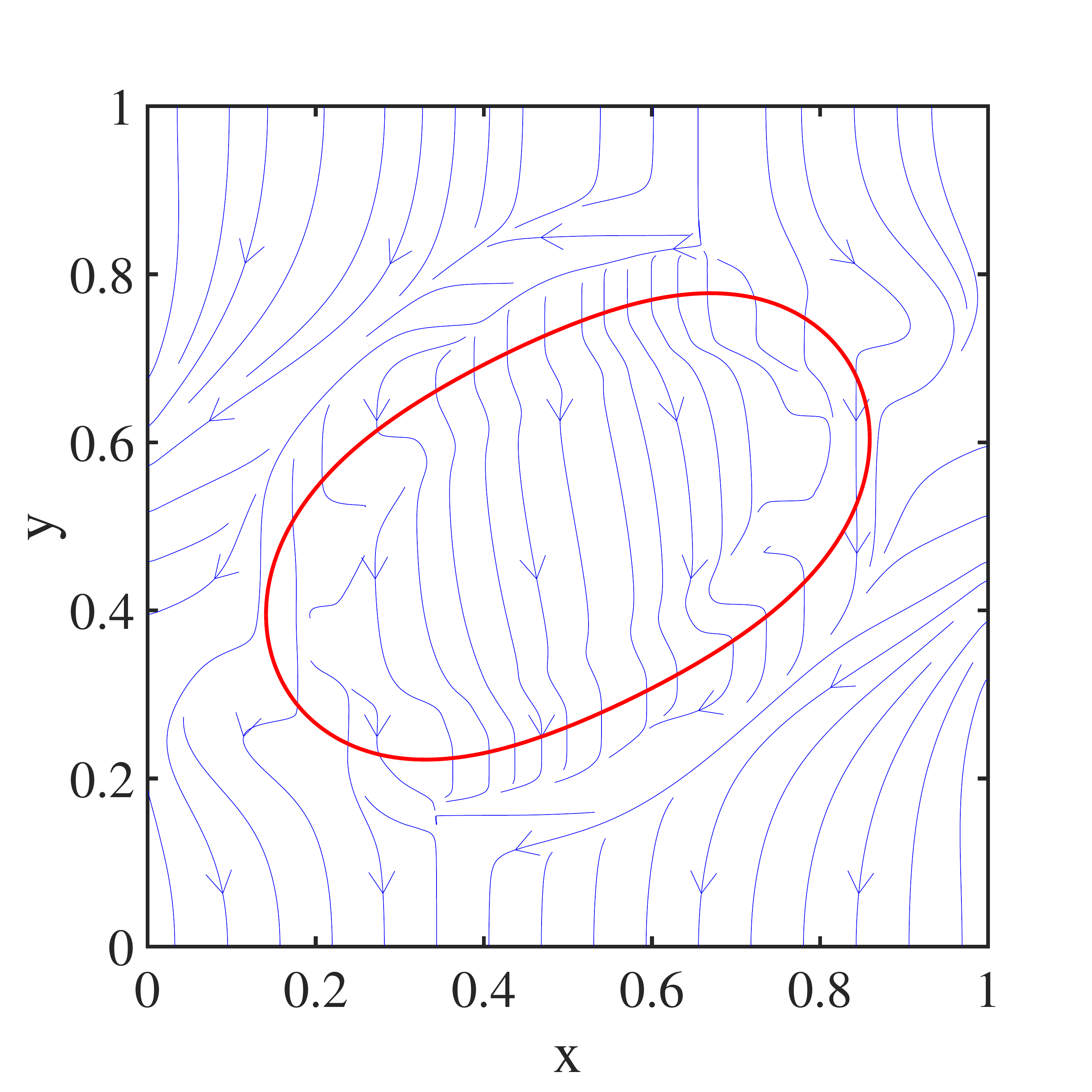} 
        \end{center}
    \end{minipage} 
    }
    \caption{Snapshot of the magnitude for flux at steady state with different permeabilities
    $K=\frac{1}{2\sigma\delta}, \frac{1}{2\sigma}$ and $\frac{\delta}{2\sigma}$.}
    \label{fig: flux fielddiagram}
\end{figure}

\begin{figure}[htbp]
    \subfigure[Higher permeability 
    ]{
    \begin{minipage}[h]{0.33\linewidth}
        \begin{center}
            \includegraphics[width=1.0\textwidth]{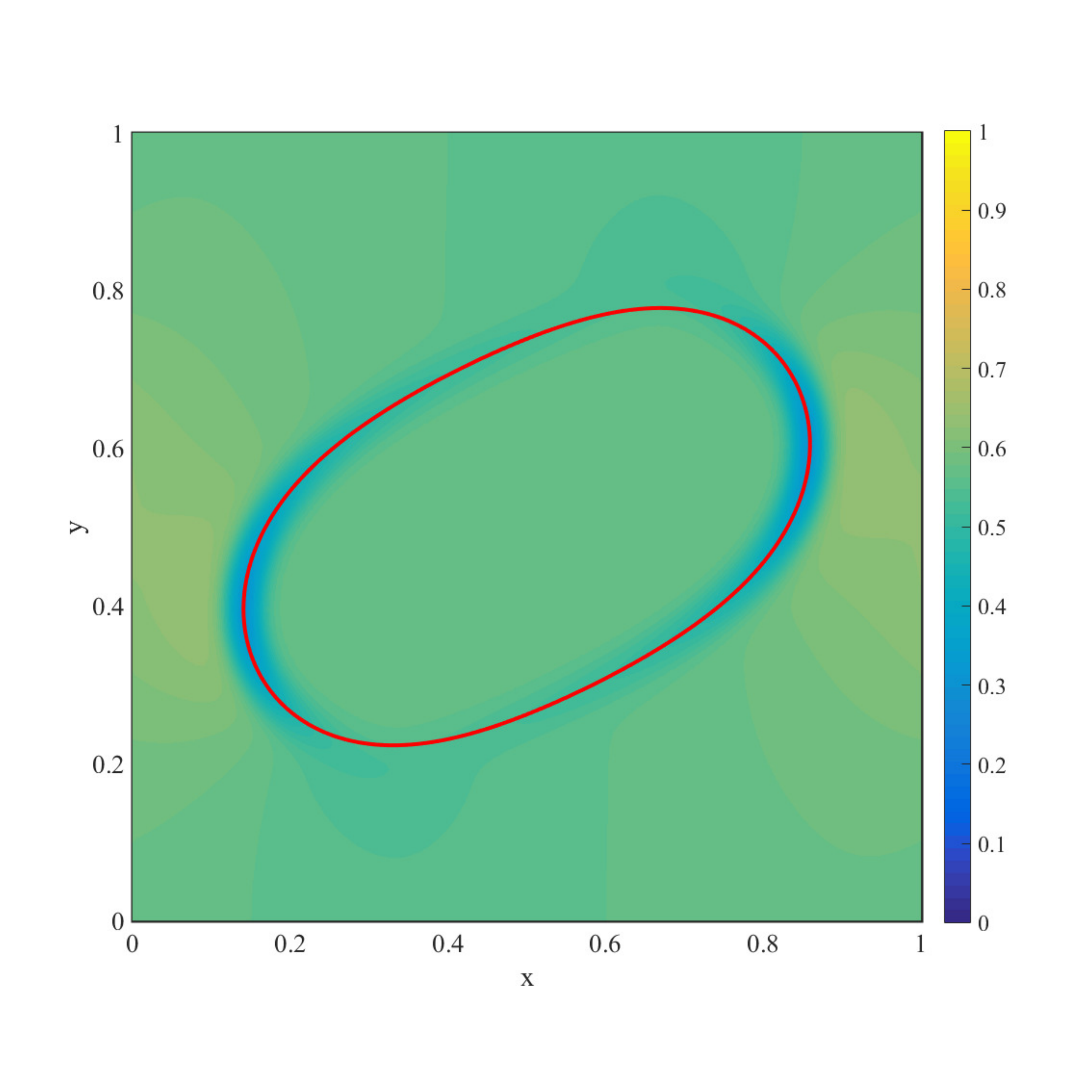} 
        \end{center}
    \end{minipage}
    }
    \subfigure[Medium permeability 
     ]{
    \begin{minipage}[h]{0.33\linewidth}
        \begin{center}
            \includegraphics[width=1.0\textwidth]{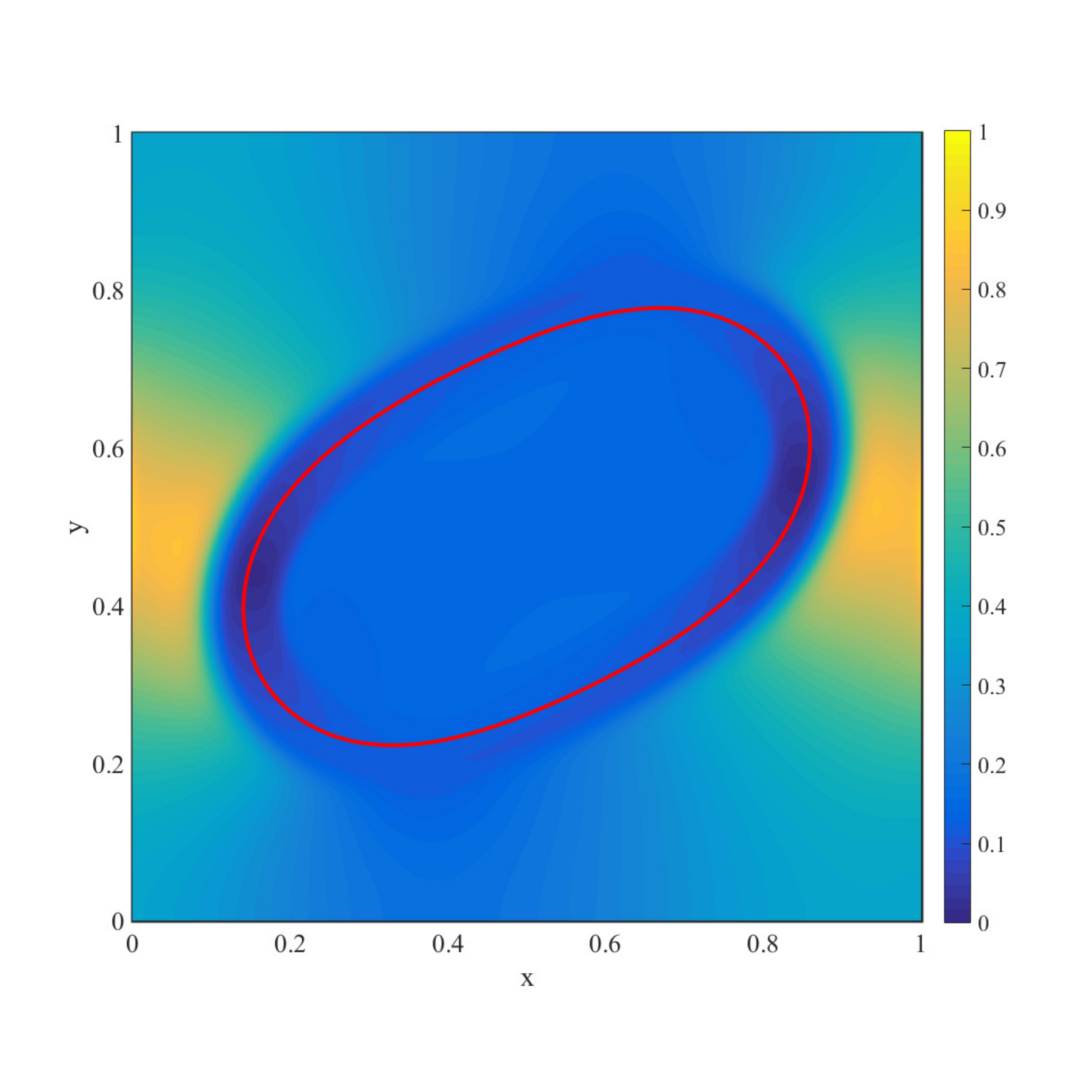} 
        \end{center}
    \end{minipage}
    }
    \subfigure[Lower permeability 
     ]{
    \begin{minipage}[h]{0.33\linewidth}
        \begin{center}
            \includegraphics[width=1.0\textwidth]{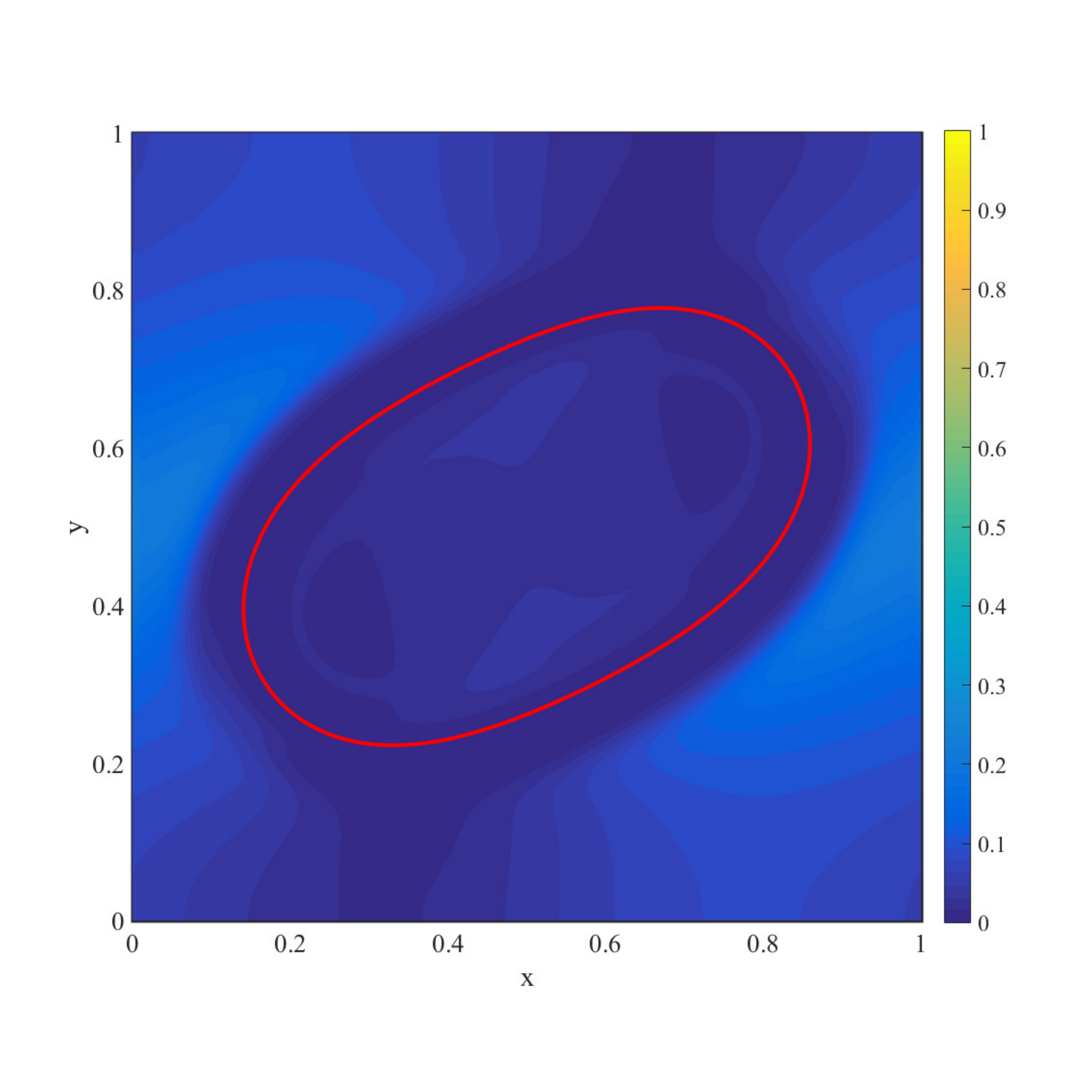} 
        \end{center}
    \end{minipage} 
    }
    \caption{Snapshot of the magnitude for flux at steady state with different permeabilities 
     $K=\frac{1}{2\sigma\delta}, \frac{1}{2\sigma}$ and $\frac{\delta}{2\sigma}$}
    \label{fig: flux magiagram}
\end{figure}

In the next, we consider there are two droplets in the shear flow in a rectangular area $\Omega=[0,2]\times[0,1]$.
The initial condition for $\phi$ is considered as 
\begin{align}
    \phi(\bm{x},0)=\tanh\big(\frac{0.2-\sqrt{(x-0.5)^{2}+(y-0.7)^{2}}}{\sqrt{2}\epsilon}\big)
    +\tanh\big(\frac{0.2-\sqrt{(x-1.5)^{2}+(y-0.3)^{2}}}{\sqrt{2}\epsilon}\big)+1.0.
\end{align}
The permeability $K$ is chosen as $K=\frac{1}{2\sigma}$.
The initial conditions of $c$ and $\bm{u}$ and all the parameters are chosen the same as in the former example $\eqref{exm: comparison}$.
As shown in Fig. \ref{fig: 2droplets}, the initial profiles of concentration and 
droplets are shown in Fig. \ref{fig: 2droplets}

\begin{figure}[htbp]
    \subfigure[$c(\bm{x},0)$.]{
    \begin{minipage}[h]{0.5\linewidth}
        \begin{center}
            \includegraphics[width=1.0\textwidth]{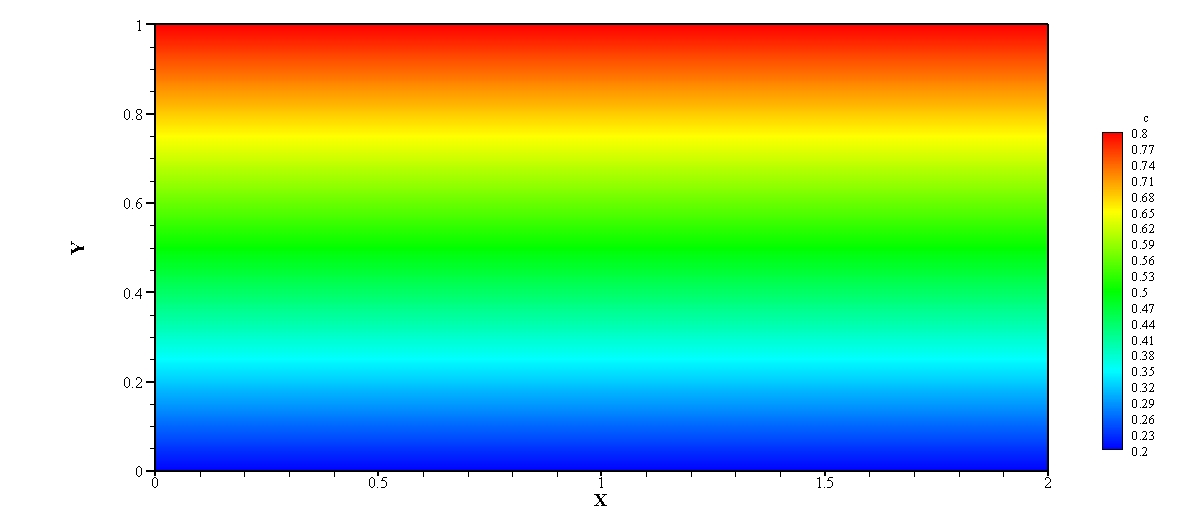} 
        \end{center}
    \end{minipage}
    }
    \subfigure[$\phi(\bm{x},0)$.]{
    \begin{minipage}[h]{0.5\linewidth}
        \begin{center}
            \includegraphics[width=1.0\textwidth]{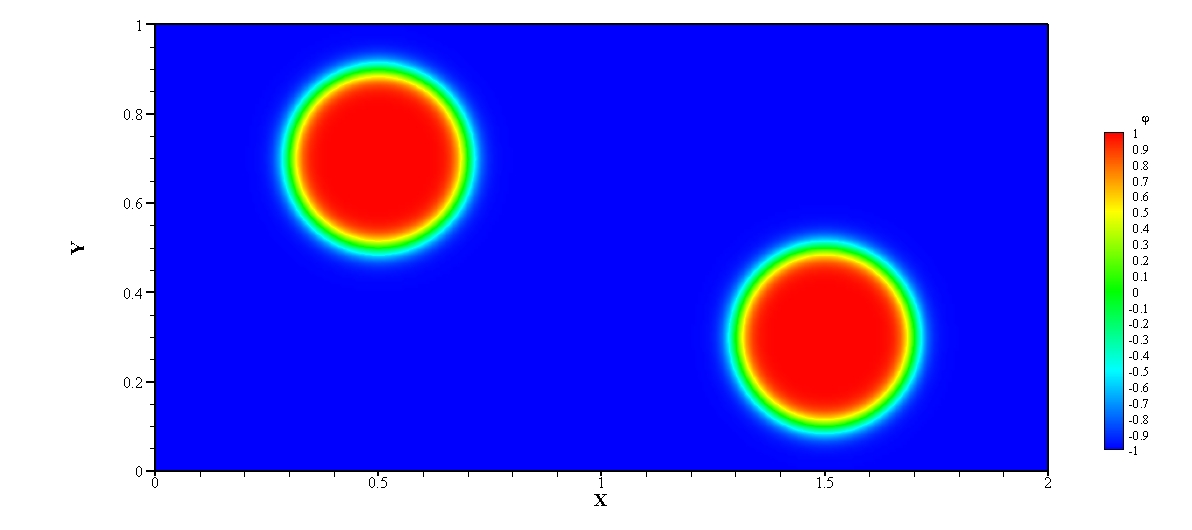} 
        \end{center}
    \end{minipage}
    }
    \caption{Profile of initial condition.}
    \label{fig: 2droplets}
\end{figure}

In Fig. \ref{fig: flux magiagram in 2drops}, the magnitude of flux overtime are presented. Due the counter flow, two droplets move closer and the merge into one droplet around time $t= 1.2$. Before $t=1.2$, the restrict diffusion are observed on two interface and the concentrations in two droplets are two distinct constants as shown in Fig. \ref{fig: concentration in 2drops} a-b. After the merge time, the   droplet achieves new equilibrium of concentration (see  Fig. \ref{fig: concentration in 2drops} d), where the inner region flux is close to zero \ref{fig: flux magiagram in 2drops}d.  
\begin{figure}
	 \begin{center}
		\includegraphics[width=1.0\textwidth]{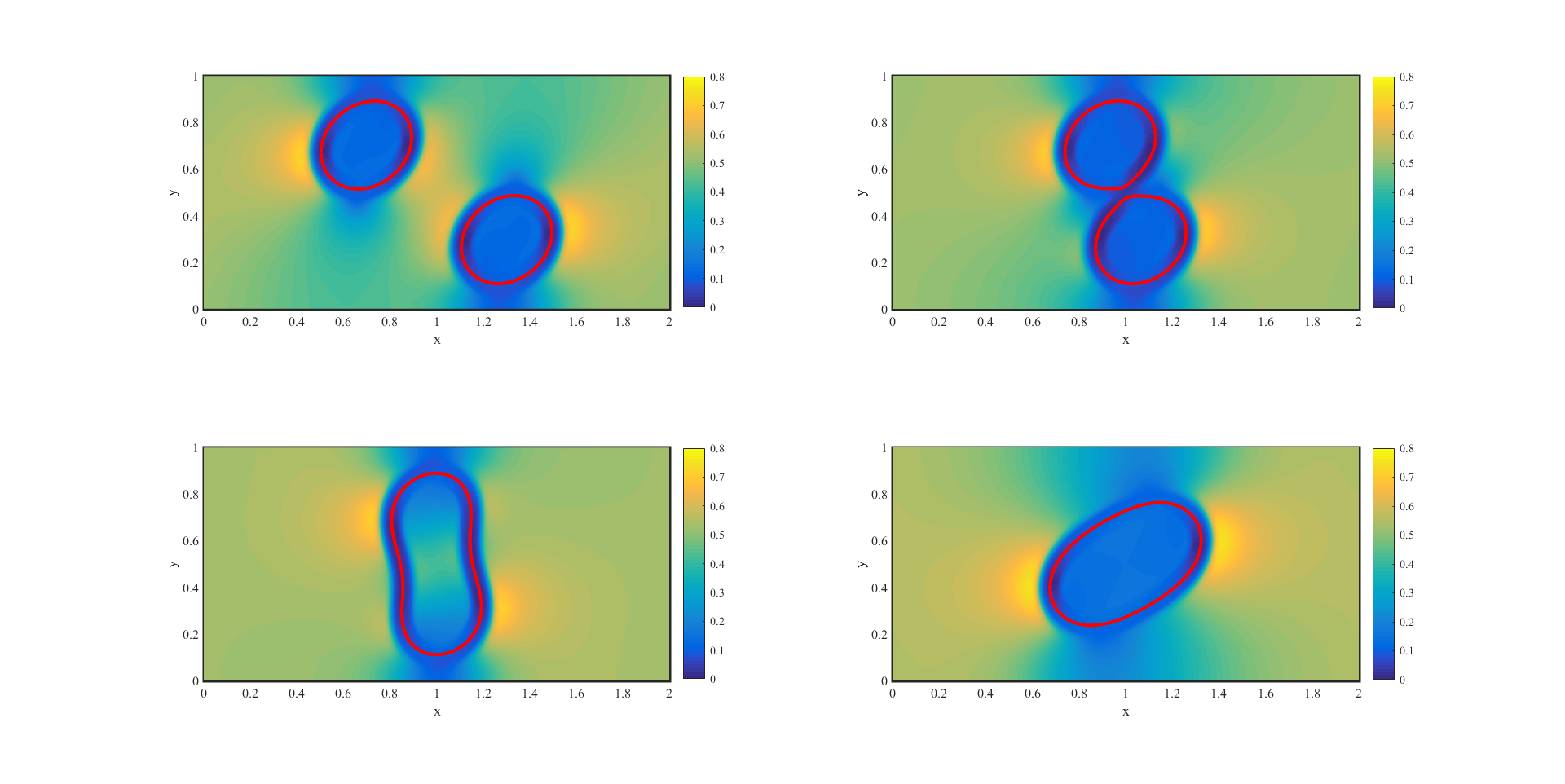} 
	\end{center}
    \caption{Evolution of the magnitude for flux at different time $t=0.5$,$t=1.1$, $t=1.2$ and $t=5.0$. The red circles show the location of membrane.}
\label{fig: flux magiagram in 2drops}
\end{figure}

\begin{figure}
	\begin{center}
		\includegraphics[width=1.0\textwidth]{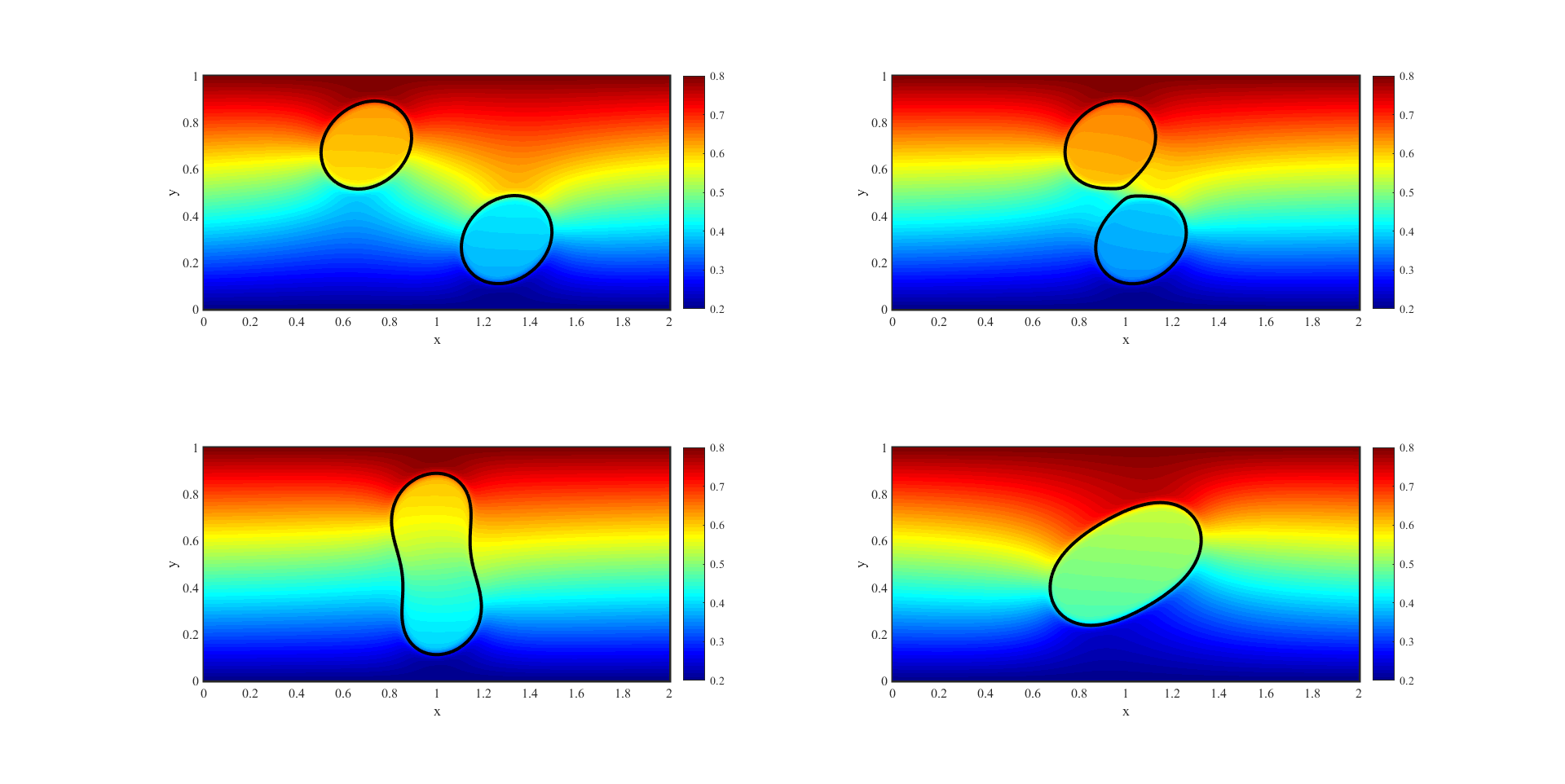} 
	\end{center}
    \caption{Evolution of the concentration at different time $t=0.5$,$t=1.1$, $t=1.2$ and $t=5.0$. The black circles show the location of membrane.}
\label{fig: concentration in 2drops}
\end{figure}

\section{Conclusions}\label{sec: conclusions}
    In this paper, we are devoted to develop a new mathematical model for mass transfer through a semi-permeable membrane with the restricted diffusion by using diffusive interface method. 
    The model is thermal dynamically consistent derived based on energy variation method.  It turns out that the restrict diffusion on the interface could be modeled by modifying the diffusion coefficient as a function of interface permeability  in the dissipation functional. 
    The sharp interface limit for the concentration function is conducted 
    by using the asymptotic analysis.  It confirms that as the interface thickness goes to zero, the diffusive interface model converges to the sharp interface model. 
    Besides, we establish a linear and unconditional energy stable numerical scheme to solve 
    the obtained nonlinear coupled system. 
    The centered-block finite difference method with stagger mesh is used for spacial discretization. 
    The numerical simulations ensure the asymptotic analysis results of the model and energy stability  of numerical scheme.  
   Finally, the validated model is used to  study to interface  permeability. It shows that our diffusive interface model could handle the restrict diffusion  problem successfully and efficiently. Especially, it could deal with the case when two droplets with semipermeable interface merge together easily.   
\section{Acknowledgement}
This work was partially supported by the 
National Natural Science Foundation of China  no. 12071190  and Natural Sciences and Engineering Research Council of Canada (NSERC). Authors also would like to thank  American Institute of Mathematics where this project  started. 
\bibliographystyle{plain}
\bibliography{references}

\begin{thebibliography}{10}

\bibitem{adalsteinsson1995fast}
David Adalsteinsson and James~A Sethian.
\newblock A fast level set method for propagating interfaces.
\newblock {\em Journal of computational physics}, 118(2):269--277, 1995.

\bibitem{CH2}
J.~W. Cahn.
\newblock {Free energy of a nonuniform system.
  \uppercase\expandafter{\romannumeral 2}. thermodynamic basis}.
\newblock {\em Journal of Chemical Physics}, 30(5):1121 -- 1124, 1959.

\bibitem{CH1}
J.~W. Cahn and J.~E. Hilliard.
\newblock {Free energy of a nonuniform system.
  \uppercase\expandafter{\romannumeral 1}. interfacial free energy}.
\newblock {\em Journal of Chemical Physics}, 28(2):258--267, 1958.

\bibitem{CH3}
J.~W. Cahn and J.~E. Hilliard.
\newblock {Free energy of a nonuniform system.
  \uppercase\expandafter{\romannumeral 3}. nucleation in a two component
  incompressible fluid}.
\newblock {\em Journal of Chemical Physics}, 31(3):688 -- 699, 1959.

\bibitem{Chorin1968NS}
A.~J. Chorin.
\newblock {Numerical solution of the Navier-Stokes equations}.
\newblock {\em Mathematics of Computation}, 22(104):745 -- 762, 1968.

\bibitem{davidson2002volume}
Malcolm~R Davidson and Murray Rudman.
\newblock Volume-of-fluid calculation of heat or mass transfer across deforming
  interfaces in two-fluid flow.
\newblock {\em Numerical Heat Transfer: Part B: Fundamentals},
  41(3-4):291--308, 2002.

\bibitem{eisenberg2010energy}
Bob Eisenberg, Yunkyong Hyon, and Chun Liu.
\newblock Energy variational analysis of ions in water and channels: Field
  theory for primitive models of complex ionic fluids.
\newblock {\em The Journal of Chemical Physics}, 133(10):104104, 2010.

\bibitem{flynn1974mass}
GL~Flynn, Samuel~H Yalkowsky, and TJ~Roseman.
\newblock Mass transport phenomena and models: theoretical concepts.
\newblock {\em Journal of Pharmaceutical Sciences}, 63(4):479--510, 1974.

\bibitem{gong2014immersed}
Xiaobo Gong, Zhaoxin Gong, and Huaxiong Huang.
\newblock An immersed boundary method for mass transfer across permeable moving
  interfaces.
\newblock {\em Journal of Computational Physics}, 278:148--168, 2014.

\bibitem{Shen2006Projection}
J.~L. Guermond, P.~Minev, and J.~Shen.
\newblock {An overview of projection methods for incompressible flows}.
\newblock {\em Computer Methods in Applied Mechanics and Engineering},
  195(44-47):6011 -- 6045, 2006.

\bibitem{Huang2009Immersed}
H.~Huang, K.~Sugiyama, and S.~Takagi.
\newblock {An immersed boundary method for restricted diffusion with permeable
  interfaces}.
\newblock {\em Journal of Computational Physics}, 228(15):5317 -- 5322, 2009.

\bibitem{Jiang2013Cellular}
H.~Jiang and S.~X. Sun.
\newblock {Cellular pressure and volume regulation and implications for cell
  mechanics}.
\newblock {\em Biophysical Journal}, 105(3):609--619, 2013.

\bibitem{johansson2005mass}
Jonas Johansson, C~Patrik~T Svensson, Thomas M{\aa}rtensson, Lars Samuelson,
  and Werner Seifert.
\newblock Mass transport model for semiconductor nanowire growth.
\newblock {\em The Journal of Physical Chemistry B}, 109(28):13567--13571,
  2005.

\bibitem{Layton2006Modeling}
A.~T. Layton.
\newblock {Modeling water transport across elastic boundaries using an explicit
  jump method}.
\newblock {\em SIAM Journal on Scientific Computing}, 28(6):2189--2207, 2006.

\bibitem{leveque1997immersed}
Randall~J LeVeque and Zhilin Li.
\newblock Immersed interface methods for stokes flow with elastic boundaries or
  surface tension.
\newblock {\em SIAM Journal on Scientific Computing}, 18(3):709--735, 1997.

\bibitem{li2019interactions}
Xin Li, Weiwen Wang, Pan Zhang, Jianlong Li, and Guanghui Chen.
\newblock Interactions between gas--liquid mass transfer and bubble behaviours.
\newblock {\em Royal Society open science}, 6(5):190136, 2019.

\bibitem{pakhomov2009lipid}
Andrei~G Pakhomov, Angela~M Bowman, Bennett~L Ibey, Franck~M Andre, Olga~N
  Pakhomova, and Karl~H Schoenbach.
\newblock Lipid nanopores can form a stable, ion channel-like conduction
  pathway in cell membrane.
\newblock {\em Biochemical and biophysical research communications},
  385(2):181--186, 2009.

\bibitem{peskin2002immersed}
Charles~S Peskin.
\newblock The immersed boundary method.
\newblock {\em Acta numerica}, 11:479--517, 2002.

\bibitem{Yang2010Stability}
J.~Shen and X.~Yang.
\newblock {Energy stable schemes for Cahn-Hilliard phase-field model of
  two-phase incompressible flows}.
\newblock {\em Chinese Annals of Mathematics, Series B}, 31:743 -- 758, 2010.

\bibitem{Shen2010stability}
J.~Shen and X.~Yang.
\newblock {Numerical approximations of Allen-Cahn and Cahn-Hilliard equations}.
\newblock {\em Discrete and Continuous Dynamical Systems}, 28(4):1669 -- 1691,
  2010.

\bibitem{Shen2015Stability}
J.~Shen and X.~Yang.
\newblock {Decoupled, energy stable schemes for phase-field models of two-phase
  incompressible flows}.
\newblock {\em SIAM Journal of Numerical Analysis}, 53(1):279 -- 296, 2015.

\bibitem{shen2020energy}
Lingyue Shen, Huaxiong Huang, Ping Lin, Zilong Song, and Shixin Xu.
\newblock An energy stable c0 finite element scheme for a quasi-incompressible
  phase-field model of moving contact line with variable density.
\newblock {\em Journal of Computational Physics}, 405:109179, 2020.

\bibitem{song2018electroneutral}
Zilong Song, Xiulei Cao, and Huaxiong Huang.
\newblock Electroneutral models for a multidimensional dynamic
  poisson-nernst-planck system.
\newblock {\em Physical Review E}, 98(3):032404, 2018.

\bibitem{unverdi1992front}
Salih~Ozen Unverdi and Gr{\'e}tar Tryggvason.
\newblock A front-tracking method for viscous, incompressible, multi-fluid
  flows.
\newblock {\em Journal of computational physics}, 100(1):25--37, 1992.

\bibitem{Waals1893Diffuse}
J.~D. van~der Waals.
\newblock {The thermodynamic theory of capillarity under the hypothesis of a
  continuous variation of density}.
\newblock {\em Journal of Statistical Physics}, 20:197--244, 1893.

\bibitem{wang2020immersed}
Xiaolong Wang, Xiaobo Gong, Kazuyasu Sugiyama, Shu Takagi, and Huaxiong Huang.
\newblock An immersed boundary method for mass transfer through porous
  biomembranes under large deformations.
\newblock {\em Journal of Computational Physics}, 413:109444, 2020.

\bibitem{welch2000volume}
Samuel~WJ Welch and John Wilson.
\newblock A volume of fluid based method for fluid flows with phase change.
\newblock {\em Journal of computational physics}, 160(2):662--682, 2000.

\bibitem{Wise2007Solving}
S.~Wise, J.~Kim, and J.~Lowengrub.
\newblock {Solving the regularized, strongly anisotropic Cahn–Hilliard
  equation by an adaptive nonlinear multigrid method}.
\newblock {\em Journal of Computational Physics}, 226(1):414--446, 2007.

\bibitem{Xu2006Stability}
C.~Xu and T.~Tang.
\newblock {Stability analysis of large time-stepping methods for epitaxial
  growth models}.
\newblock {\em SIAM Journal of numerical anaysis}, 44:1759 -- 1779, 2006.

\bibitem{xu2018osmosis}
Shixin Xu, Bob Eisenberg, Zilong Song, and Huaxiong Huang.
\newblock Osmosis through a semi-permeable membrane: a consistent approach to
  interactions.
\newblock {\em arXiv preprint arXiv:1806.00646}, 2018.

\bibitem{xu2017sharp}
Xianmin Xu, Yana Di, and Haijun Yu.
\newblock Sharp-interface limits of a phase-field model with a generalized
  navier slip boundary condition for moving contact lines.
\newblock {\em Journal of Fluid Mechanics}, 849:805–833, 2018.

\bibitem{Yue2004Diffuse}
Pengtao Yue, James~J Feng, Chun Liu, and Jie Shen.
\newblock A diffuse-interface method for simulating two-phase flows of complex
  fluids.
\newblock {\em Journal of Fluid Mechanics}, 515:293, 2004.

\bibitem{Zhang2018ATP}
H.~Zhang, Z.~Shen, B.~Hogan, A.~I. Barakat, and C.~Misbah.
\newblock {ATP release by red blood cells under flow: model and simulations}.
\newblock {\em Biophysical Journal}, 115(11):2218--2229, 2018.

\bibitem{Zhu1999Coarsening}
J.~Zhu, L.~Chen, J.~Shen, and V.~Tikare.
\newblock {Coarsening kinetics from a variable-mobility Cahn-Hilliard equation:
  Application of a semi-implicit Fourier spectral method}.
\newblock {\em Physical Review E}, 60(4):3564 -- 3572, 1999.

\end{thebibliography}
\end{document}